\documentclass[12pt]{article}
 
\usepackage{amstext,amssymb,amsmath,stmaryrd,amsthm}
\usepackage{geometry} % see geometry.pdf on how to lay out the page.  
\usepackage{color}
\usepackage{hyperref}
\usepackage{makeidx}
\makeindex

\numberwithin{equation}{section}

\newtheorem{theorem}{Theorem}[section]
\newtheorem{cor}[theorem]{Corollary}
\newtheorem{Lemma}[theorem]{Lemma}
\newtheorem{proposition}[theorem]{Proposition}

\newtheorem{rem}[theorem]{Remark}

\begin{document}

\title{Concentration inequalities for Gibbs measures.}
\author{Ioannis Papageorgiou \thanks{Institut de Math\'ematiques de Toulouse, Universit\'e Paul Sabatier, 31062 Toulouse cedex 09, France. Email: ioannis.papageorgiou@math.univ-toulouse.fr}}

\date{}
\maketitle

\begin{abstract}
We are interested in Sobolev type inequalities and their relationship with concentration properties on higher dimensions. We consider unbounded spin systems on the d-dimensional lattice with interactions that increase slower than a quadratic. At first we assume that the one site measure satisfies a Modified log-Sobolev inequality with a constant uniformly on the boundary conditions and we determine conditions so that the infinite dimensional Gibbs measure satisfies a concentration  as well as a Talagrand type inequality, similar to the ones obtained by Barthe and Roberto [B-R2] for the product measure.
Then a Modified Log-Sobolev type  concentration property is obtained under weaker conditions referring to the  Log-Sobolev inequalities for the boundary free measure.
\end{abstract}

\section{Introduction.}
We focus
on the Modified Logarithmic Sobolev inequality (see [G-G-M1]) as well as on  Log-Sobolev inequalities  [G] for unbounded spin systems on the d-dimensional lattice with interactions that increase slower than a quadratic. We investigate two problems related with concentration properties. The first, presented in Section 2, relates to the Modified Log-Sobolev inequality for the one site measure with boundary conditions. Suppose  that the one site measure satisfies a Modified log-Sobolev inequality with a constant uniformly on the boundary conditions. The aim of this work is to determine conditions  under which the infinite dimensional Gibbs measure satisfies a concentration inequality (Theorem \ref{theorem1}). As a consequence, a Talagrand type inequality is also obtained (Corollary \ref{cor1}).
The second, presented in Section 3, is a perturbation problem  that refers  to the  Log-Sobolev inequality for the boundary free one site measure. Suppose  that the boundary free one site (non log-concave) measure satisfies a  Log-Sobolev inequality.  The aim   here is to determine conditions  under which if we perturbe the measure with interactions,  the infinite dimensional Gibbs measure of the corresponding local specification satisfies a concentration inequality (Theorem \ref{NStheorem1}). Here too, as a consequence, a Talagrand type inequality is also obtained (Corollary \ref{NScor1}).

 The problem of concentration of measure properties for product measures that satisfy a Modified Log-Sobolev inequality has been recently examined (see for instance [G-G-M1] and [B-R2]). In addition the problem of the stronger  Log-Sobolev inequalities for the infinite dimensional Gibbs measure has been addressed (e.g. [G-Z], [O-R] and [I-P]), as well as the relationship of Logarithmic Sobolev inequalities with concentration properties (e.g. [B-Z] and [B-L1]). 

In this section we present the most important notions and developments related to the Sobolev type inequalities and the Concentration of measure properties.

\textbf{\textit{Sobolev type inequalities.}} The Logarithmic Sobolev   inequalities have been extensively investigated. For $1<q\leq 2$, a measure $\mu$ on $\mathbb{R}^n$ satisfies the $q$ Log-Sobolev inequality if
there exists a positive constant $C$ such that$$Ent_{\mu}(\vert f\vert^q)\leq C\mu \left\vert \nabla f\right \vert ^q  \   \   \   \    \   \  \   \    (LSq)$$
where $Ent_\mu (f)=\mu\left(f\log\frac{f}{\mu f}\right)$ and $\left\vert \nabla f\right \vert $ is the Euclidean length of the gradient $\nabla f$ of the function $f$.  The Logarithmic Sobolev inequalities where first introduced for $q=2$ by [G] and later for $1<q<2$
by [B-Z]. If a measure $\mu$ satisfies  the  q Log-Sobolev inequality, then it also satisfies the q Spectral Gap inequality, that is
$$\mu \vert f-\mu f\vert^q\leq C'\mu \left\vert \nabla f\right\vert^q \   \   \   \   \   (SGq)$$
for some constant $C'$.
 
Recently a lot of attention has been focused on inequalities that interpolate between the  Log-Sobolev inequality and the Spectral Gap inequality (both in the classical sense of $q=2$). For a detailed account of these developments one can look on [G-G-M1] and [B-R2].  A first example of an  inequality interpolating between the Log-Sobolev and the Spectral Gap was introduced  by [Be] and then studied by [L-O]  and [B-R1].  In this paper we are interested in the Modified Log-Sobolev inequality  introduced by [G-G-M1], (see also [G-G-M2] and [B-R2]). Before we give the definition of the inequality we first  present some definitions. A function $\Phi:\mathbb{R}\rightarrow[0,\infty)$ is called a \textit{Young function} if it is convex, even, $\Phi(0)=0$ and $\lim_{x\rightarrow\infty}\Phi(x)=+\infty$.
We define the \textit{conjugate} function $\Phi^*$ to be the Legendre transform of $\Phi$,  $\Phi^*(y)=\sup_{x\geq0}\left\{ x\vert y \vert-\Phi(x)\right\}$. Furthermore, a young function $\Phi$ is called  \textit{nice} if $\lim_{x\rightarrow \infty}\frac{\Phi (x)}{x}=\infty$, $\Phi(x)=0\iff x=0$ and $\Phi'(0)=0$.
One can consider the following example of a nice  Young function, $\Phi(x)=\frac{\vert x\vert ^p}{p}$ with the conjugate being  $\Phi ^*(x)=\frac{\vert x \vert ^q}{q}$ where $p>1$ and $\frac{1}{p}+\frac{1}{q}=1$. Given a nice Young function $\Phi:\mathbb{R}\rightarrow\mathbb{R}^+$we can define the \textit{modification} of $H_{\Phi}$ of $\Phi$ to be
\begin{align*}H_{\Phi}(x)=\begin{cases}x^{2} & \text{if \ } \vert x\vert \leq 1 \\
\frac{\Phi(\vert x \vert)}{\Phi(1)}& \text{if \ }\vert x\vert \geq\ 1 \end{cases}
\end{align*}
The definition of the  MLS$(H_\Phi)$ inequality follows.

\noindent
\textit{The Modified  Log-Sobolev  Inequality on $\mathbb{R}^n$.} 
We say that the measure $\nu$ on $\mathbb{R}^n$ satisfies the Modified Log-Sobolev Inequality if there exists a positive  constant $C_{MLS}$   such that for any function $f\in C^\infty$  the following holds
$$\nu\vert f\vert^2 log\frac{\vert f\vert^2}{\nu\vert f\vert^2}\leq C_{MLS}
\int  \sum_{i=1}^n H_\Phi\left(\frac{ \nabla _if}{f}\right)f^2d\nu  \   \   \   \   \   \   \ (MLS(H_\Phi))$$
  The following remark summarises some of the main properties of the Modified Logarithmic Sobolev inequalities and (LSq) inequalities.
\begin{rem}\label{remark0}\rm(i)  (MLS), (LSq) and (SGq) are stable under tensorisation: Suppose $\mu_1$ and $\mu_2$ satisfy $(MLS)/(LSq)/(SGq)$ inequalities with constants $c_1$ and $c_2$ respectively.  Then $\mu_1\otimes\mu_2$ satisfies an $(MLS)/(LSq)/(SGq)$ inequality with constant $\max\{c_1, c_2\}$.

\rm{(ii) \rm$(MLS)/(LS2)\Rightarrow (SG2)$:} Suppose $\mu$ satisfies an $(MLS)/(LS2)$ inequality with constant $c$.  Then $\mu$ satisfies a  $SG_2$ inequality  with constant $c_0\leq \frac{c}{2}$ i.e.
\[
\mu\left|f - \mu f\right|^2 \leq c_{0} \mu\left(|\nabla f|^2\right)
\]
for all smooth $f$.
\end{rem}
The proof of the remark can be found in [A-B-C].

\textit{\textbf{Concentration of Measure Phenomenon.}} An  important future of Sobolev type inequalities is their relationship with concentration of measure properties. The concentration inequality was first investigated in [T1] and [T2], as well as in [B-L1], while the case of the Log-Sobolev q inequalities was investigated in [B-L2] and [B-Z].
 
Consider a measure $\mu$ on $\mathbb{R}^n$ that satisfies the (LSq) or (MLS($H_\Phi$)) inequality. Then for every Lipschitz function on $\mathbb{R}$ with $\Vert F \Vert _{Lip}\leq 1$, for some constant K and any $r>0$
\begin{align*}\mu \left( F- \mu F\geq r\right)\leq e^{-KE(r)}
\end{align*}
with $E(r)=r^\frac{q}{q-1}$ in the case where $\mu$ satisfies the (LSq) inequality (see [B-L2] and [B-Z]), as well as the (MLS$H_{\vert x \vert^q}$) inequality (see [G-G-M1]). When $\mu$ satisfies an (MLS($H_\Phi$)) inequality  $E(r)=r\omega_{H}^*(r)$ (see [B-R2]).

In particular, concerning the Modified Log-Sobolev inequality MLS($H_\Phi$) for the  product measure on $\mathbb{R}^n$ and its relationship with concentration of measure,  in [B-R2] the following theorem was shown
\begin{theorem}\label{Theorem[B-R2]}\textbf{[B-R2]} Let $\mu$ be a probability measure on $\mathbb{R}$, which we assume to be absolutely continuous with respect to Lebesgue's measure. Let $H:\mathbb{R}\rightarrow \mathbb{R}^+$ be an even convex function, with $H(0)=0$. Assume that  $x\rightarrow \frac{H(x)}{x^2}$ is non-decreasing for $x>0$ and that $H^*$ is strictly convex. If $\mu$ satisfies an $MLS(H)$ with constant $k$,  then for every  Borel set $A\subset \mathbb{R}^n$ with $\mu^n(A)\geq\frac{1}{2}$ 
\begin{align}1-\mu^n\left( A+\left\{ x:\sum_{i=1}^nH^*(x_1)<r\right\}  \right)\leq e^{-Kr}    \   \    \forall r\geq 0
\end{align}where $K=\omega_H(2)k\omega^*_H(\frac{1}{\omega_H(2)k})$.
\end{theorem}
As a consequence, the following Talagrand type inequality was then obtained
 \begin{equation}\label{Talagrand}\mu ^n\left(A+\sqrt{r}B_2+\frac{1}{\omega_{\Phi^*}^{-1}(\frac{1}{r})}B_{\Phi^*}\right)\geq 1-e^{-Cr}
\end{equation}
where we denote $B_{\Phi^*}:=\{x:\sum_{i=1}^n\Phi^*(x_i)<1\}$.

\textit{\textbf{Infinite dimensional analysis.}} An interesting question is to determine conditions so that the infinite dimensional Gibbs measure satisfies the  Log-Sobolev inequality.  In the case where the interactions are quadratic this problem has been extensively studied in the case $q=2$, for example in [Ba], [B-E], [B-H], [G-Z],   [L], [O-R], [Z1] and [Z2].  The case $1<q\leq2$ was looked at in [I-P], while the case of the  Spectral Gap inequality was addressed in [G-R].

The Infinite Dimensional Setting: We consider the  d-dimensional
Lattice $\mathbb{Z}^d$ for some
$d \in \mathbb{N}$.  For any subset $\Lambda$ we denote $\left\vert \Lambda \right\vert$ the cardinality of $\Lambda$. When $\left\vert \Lambda \right\vert<+\infty$,
we will write $\Lambda \subset \subset \mathbb{Z}^d $.  We consider continuous unbounded random variables in $\mathbb{R}$, representing spins or particles. Our configuration space is $\Omega=\mathbb{R}^{\mathbb{Z}^d}$. For any
$\omega\in\Omega$ and $\Lambda \subset\subset \mathbb{Z}^d$ we denote
$\omega=(\omega_i)_{i\in \mathbb{Z}^d}, \omega_{\Lambda}=(\omega_i)_{i\epsilon\Lambda} \text{\; and \;}   \omega_{\Lambda^c}=(\omega_i)_{i\epsilon\Lambda^c} $ 
where $\omega_i \in \mathbb{R}$. When $\Lambda=\{i\}$ we will write $\omega_i:=\omega_{\{i\}}$.
 We consider integrable functions $f$ that depend on a finite set of variables
$(x_i)_{\Sigma_f}$ for  finite subset $\Sigma_f\subset\subset \mathbb{Z}^d$.
 For any subset $\Lambda\subset\subset
\mathbb{Z}^d$ we define the probability measure
\begin{equation}\label{locspec}\mathbb{E}^{\Lambda,\omega}(dX_\Lambda) =
\frac{e^{-H^{\Lambda,\omega}}dX_\Lambda} {Z^{\Lambda,\omega}} \end{equation}
where
 \begin{itemize}
\item $X_{\Lambda}=(x_i)_{i\epsilon\Lambda}$, $dX_\Lambda=\prod_{i\epsilon\Lambda}dx_i$ and $Z^{\Lambda,\omega}=\int e^{-H^{\Lambda,\omega}}dX_\Lambda$
\item $i\sim j$ means that the nodes $i$ and $j$   on $\mathbb{Z}^d$ are neighbours
\item 
$H^{\Lambda,\omega}=\sum_{i\epsilon\Lambda}
\phi (x_i)+\sum_{i\epsilon\Lambda,j\sim
i}J_{ij}V(x_i
,z_j)$,  for $J_{ij}$ constants and $V\in C^\infty$.
\end{itemize}
\noindent
and
\begin{itemize}
\item
$z_j=(x_{\Lambda}\circ\omega_{\Lambda^c})_j=\begin{cases}x_j & ,j\epsilon\Lambda  \\
\omega_j & ,j\notin\Lambda \\
\end{cases}$
\end{itemize}
We call $\phi$ the phase and $V$ the potential of the interaction.
We assume that $\vert J_{i,j}\vert\in [0,J_0]$ for some positive $J_0$. When $\Lambda=\{ i \}$ we will also write
$\mathbb{E}^{i,\omega}:=\mathbb{E}^{\{i \},\omega}$. Furthermore, we denote $\{\sim i\}:=\{j\in \mathbb{Z}^d:j\sim i\}$, while for any $\Lambda\subset \mathbb{Z}^d$, $\sim \Lambda:=\{j\in \mathbb{Z}^d\smallsetminus\Lambda:j\sim i \text{ \ for some \ }i\in \Lambda\}$. We will denote   $$\mathbb{E}^{\Lambda,\omega}f=\int f d \mathbb{E}^{\Lambda,\omega}(X_\Lambda)$$   The Gibbs measure $\mu$ for the local specification
$\{\mathbb{E}^{\Lambda,\omega}\}_{\Lambda\subset\subset \mathbb{Z}^d,\omega \in
\Omega}$ is defined as the  probability measure which solves the DLR equation
$$ \mu\mathbb{E}^{M,\ast} f=\mu f $$ 
for all finite  sets  $\Lambda\subset\subset \mathbb{Z}^d$ and bounded measurable functions $f$ on $\Omega$ (see [Pr]). For criterions on the existence and uniqueness of the Gibbs measure one can look at [D] and [B-HK].  For any subset $\Lambda\subset \mathbb{Z}^d$ and $\nabla_{i}=\frac{\partial}{\partial x_i}$ we define the gradient
$$\left\vert \nabla_{\Lambda} f
\right\vert^q=\sum_{i\in\Lambda}\left\vert \nabla_{i} f
\right\vert^q \  \ \text{for \;} q \in (1,2] $$ and the higher dimensional analogue of the function $H_\Phi$ as $$H_{\Phi}\left(\frac{\nabla_\Lambda f}{f}\right)=\sum_{i\in \Lambda}H_{\Phi}\left(\frac{\nabla_i f}{f}\right)$$
When $\Lambda=\mathbb{Z}^d$ we will simply write $\nabla:=\nabla_{\mathbb{Z}^d}$. For this setting the analogues of the Modified Logarithmic Sobolev and Spectral Gap inequalities can be defined. In [I-P]   concerning the (LSq) $q\in (1,2]$ inequality for the infinite dimensional Gibbs measure the following theorem was shown. \begin{theorem}\label{Theorem[I-P]} \textbf{([I-P])} Suppose that the local specification $\{\mathbb{E}^{\Lambda,\omega}\}_{\Lambda \subset\subset \mathbb{Z}^d,\omega\in \Omega}$ as in (\ref{locspec}) has interactions $V$  such that $\left\Vert \nabla_i \nabla_j V(x_i,x_j) \right\Vert_{\infty}<\infty$ and  $\mathbb{E}^{\{ i \},\omega}$  satisfies an (LSq) inequality for some positive constant, uniformly on the boundary conditions $\omega\subset \Omega$. Then, for sufficiently small $J_0$, the corresponding infinite dimensional Gibbs measure $\mu$ is unique and satisfies an  (LSq) inequality.\end{theorem}  Furthermore, from [B-Z], for every $f:\Omega \rightarrow \mathbb{R}$  such that $\Vert \left\vert \nabla f
\right\vert^q\Vert_\infty<1$ the following concentration property  holds
\begin{equation*}\mu \left\{\left\vert  f-\int fd\mu\right\vert\geq r\right\}\leq 2\exp\left\{-\frac{(q-1)^p}{C^{p-1}}r^p\right\}
\end{equation*}
for all $r>0$, where $\frac{1}{p}+\frac{1}{q}=1$ and $C>0$.
  \section{Modified LS and Concentration for the Gibbs measure.}  In the previous section we discussed the importance of the Modified Log-Sobolev inequalities as interpolation inequalities between the Log-Sobolev  and the Spectral Gap inequalities. Furthermore, we saw how  the different types of Sobolev  inequalities are related  to concentration of measure properties and the Talagrand inequalities.
 
At the end we discussed how the (LS) inequality in the case of quadratic interactions can be extended from the one dimension to the infinite dimensional Gibbs measure, which then implies the analogous concentration properties.
 
The question that arises is whether similar results for the infinite dimensional Gibbs measure can be obtained in the case of the Modified Log-Sobolev inequality. If the Modified Log-Sobolev inequality is proven to be true for the infinite Gibbs measure then automatically the appropriate concentration of measure properties will follow.
 
Although, a result similar with the one obtained for the (LS) in  [I-P] would had been desirable, it appears that the methodology followed to obtain the q Log-Sobolev inequality for the Gibbs measure in [I-P] when interactions are quadratic cannot be applied directly in the case of the Modified Log-Sobolev inequality. If this was the case, then we could obtain as a byproduct the concentration of measure as well. However, although the inequality itself cannot be shown for the infinite dimensional Gibbs measure, the weaker  concentration property for the Gibbs measure can be directly obtained as shown on the  Theorem \ref{theorem1}. In this way, we show that the result of [B-R2] for the product measure holds also for the infinite dimensional Gibbs measure. We will work with the following hypotheses: 
\begin{itemize} 
\item[\textbf{(H0):}] The interaction $V$ is such that $\left\Vert \nabla_i \nabla_j V(x_i,x_j) \right\Vert_{\infty}<\infty.$
 
\item[\textbf{(H1):}] The one dimensional single site measures $\mathbb{E}^{i,\omega}$ satisfy $(MLS(H_\Phi))$ inequality with a constant $c$ which is independent of the boundary conditions $\omega$.

\item[\textbf{(H2):}]  $\frac{H_\Phi(x)}{x^2}$ is non-decreasing on $(0,+\infty)$ and there exists a $t>2$ such that   $\frac{H_\Phi(x)}{x^t}$ is non-increasing on $(0,+\infty)$.
\end{itemize}
The main result follows.
\begin{theorem} \label{theorem1}Assume the local specification $\{\mathbb{E}^{\Lambda,\omega}\}_{\Lambda \subset\subset \mathbb{Z}^d,\omega\in \Omega}$ as in (\ref{locspec})  is such that   (H0),  (H1) and (H2) hold.  Then, for sufficiently small $J_0$, the corresponding infinite dimensional Gibbs measure $\mu$ is unique and there exists an $R>0$ such that  for  every Borel set $A\subset \Omega$ with $\mu (A)\geq\frac{1}{2}$ \begin{align}\label{lastrel}1-\mu\left( A+\left\{ x:\sum_{i\in N}H_\Phi^*(x_i)<r\right\}  \right)\leq e^{-\hat Kr}  \text{,  \   for every\ } r>R
\end{align}for any $N\subset \subset \mathbb{Z}^d$ and some $\hat K>0$.\end{theorem}
As a consequence of this, the analogue of the (\ref{Talagrand}) Talagrand type inequality follows (see [B-R2]). \begin{cor}\label{cor1}Assume that the local specification $\{\mathbb{E}^{\Lambda,\omega}\}_{\Lambda \subset\subset \mathbb{Z}^d,\omega\in \Omega}$ as in (\ref{locspec}) is  such that   (H0),   (H1) and (H2) hold. Then, for sufficiently small $J_0$, the corresponding infinite dimensional Gibbs measure $\mu$ is unique and there exists an $R>0$ such that  for  every  Borel set $A\subset \Omega$ with $\mu(A)\geq\frac{1}{2}$,   $$\mu \left(A+\sqrt{r}B_2+\frac{1}{\omega_{\Phi^*}^{-1}(\frac{1}{r})}B_{\Phi^*}\right)\geq 1-e^{-Cr}$$ for $r>R$ and $C>0$.
\end{cor}

We will  frequently write $H:=H_\Phi$. Define the  sets $\Lambda_0=\{0\}$ and $\Lambda_k=\sim \Lambda_{k-1}=\{j\in \mathbb{Z}^d\text{\ s.t. \ }j\sim i \ \text{for   }i\in \Lambda_{k-i}\}$. The cardinality of the sets is $\vert \Lambda_k \vert\leq (2d)^k$. We also define 
\begin{align*}
&\Gamma_0=\cup_{k\text{\ even}}\Lambda_k=\cup\{j \in \mathbb{Z}^d : dist(j,(0,0))=2m \text{\; for some \;}m\in\mathbb{N} \cup \{0\}\}, \\ 
&\Gamma_1=\cup_{k\text{\ odd}}\Lambda_k=\mathbb{Z}^d \smallsetminus\Gamma_0  \end{align*}
Note that $dist(i,j)>1$ for all $i,j \in\Lambda_k,k\in \mathbb{N}$ and $\Gamma_0\cap\Gamma_1=\emptyset$.  Moreover $\mathbb{Z}^d=\Gamma_0\cup\Gamma_1$.  For the sake of notation, we will write $\mathbb{E}^{\Lambda_k}=\mathbb{E}^{\Lambda_k,\omega}$ for $k\in \mathbb{N}$. We will  define
$$ \mathcal{B}^{n,s}=\mathbb{E}^{\Lambda_n,\omega}\mathbb{E}^{\Lambda_{n-1},\omega}...\mathbb{E}^{\Lambda_{s+1},\omega}\mathbb{E}^{\Lambda_{s},\omega}$$
for $0\leq s\leq n$.   
\begin{Lemma}\label{lem2.1} Assume that the local specification $\{\mathbb{E}^{\Lambda,\omega}\}_{\Lambda \subset\subset \mathbb{Z}^d,\omega\in \Omega}$ as in (\ref{locspec}) is  such that   (H0), (H1) and  (H2) hold. Then, for sufficiently small $J_0$, for any measurable function $G$ and $f:\mathbb{R}^N\rightarrow \mathbb{R}$,  for $N \subset \subset \mathbb{Z}^d$, we have
\begin{equation*}\mu \left(H(  \nabla_{\Lambda_{k+1}}(\mathcal{B}^{k,s}\ f)
)G\right)\leq C_2^{k-s}C_1\mu\left(\sum_{i\in N} H( \nabla_{i} f )\left(\mathbb{E}^{\Lambda_{s}}...\mathbb{E}^{\Lambda_{k}}G\right)\right)\end{equation*}
   for any   $k\geq s\in \mathbb{N}$  such that $N\subset \Lambda_{s-1}\cup\Lambda_{s}$ and  constants $C_1>0$ and $0<C_2<1$. \end{Lemma}
\begin{proof}
Using properties of the Gibbs measure for $k\in\mathbb{N}$ we can write\begin{align*}\mu\left( H( \nabla_{\Lambda_{k+1}}\mathbb{E}^{\Lambda_{k}}f)G\right) &=\mu\left( H( \nabla_{\Lambda_{k+1}}\mathbb{E}^{\Lambda_{k}}f)\mathbb{E}^{\Lambda_{k}}G\right) \end{align*}   In order to bound the last quantity we will use the following corollary.
\begin{cor} \label{cor3.2}
  Assume that the local specification  $\{\mathbb{E}^{\Lambda,\omega}\}_{\Lambda\subset\subset \mathbb{Z}^d,\omega \in\Omega}$ satisfies (H0), (H2) and   (H1)/(H3). Then, for sufficiently small $J_0$, there exist constants $D>0$ and $0<\eta<1$ such that the following bound holds
\begin{align*}H( \nabla_{\Lambda_{k+1}}\mathbb{E}^{\Lambda_{k}}f) \leq D\sum_{i\in \Lambda_{k+1}}\mathbb{E}^{\{\sim i\}\cap \Lambda_k}H( \nabla_{i} f )+\eta\sum_{i\in \Lambda_{k}}\sum_{j\sim i}\mathbb{E}^{\{\sim j\}\cap \Lambda_k}H( \nabla_{i} f
) \end{align*} for all functions $f\in C^\infty$.
\end{cor} 
The proof of this corollary will  be presented in section \ref{convergence}.
This will give 
\begin{align}\label{revision1eq2.2new}\nonumber\mu\left( H( \nabla_{\Lambda_{k+1}}\mathbb{E}^{\Lambda_{k}}f)G\right)   \nonumber \leq & D\sum_{i\in \Lambda_{k+1}}\mu( \mathbb{E}^{\{\sim i\}\cap \Lambda_k}H( \nabla_{i} f )\mathbb{E}^{\Lambda_{k}}G)\\  \nonumber&+\eta\sum_{i\in \Lambda_{k}}\sum_{j\sim i}\mu( \mathbb{E}^{\{\sim j\}\cap \Lambda_k}H( \nabla_{i} f
)\mathbb{E}^{\Lambda_{k}}G) \\    \leq& D\sum_{i\in \Lambda_{k+1}}\mu\left( H( \nabla_{i} f )\mathbb{E}^{\Lambda_{k}}G\right)+2d\eta\sum_{i\in \Lambda_{k}}\mu\left( H( \nabla_{i} f
)\mathbb{E}^{\Lambda_{k}}G\right)\end{align}
 If we apply (\ref{revision1eq2.2new}) $k-s$  times  we will obtain \begin{align*}\mu \left(H (\nabla_{\Lambda_{k+1}}\mathcal{ B}^{k,s}f)G\right) \leq \nonumber &(2d \eta)^{k-s-1}D\sum_{i\in \Lambda_{s-1}}\mu\left( H( \nabla_{i} f )\left(\mathbb{E}^{\Lambda_{s}}...\mathbb{E}^{\Lambda_{k}}G\right)\right)\\&+(2d\eta)^{k-s}\sum_{i\in \Lambda_{s}}
\mu\left( H( \nabla_{i} f )\left(\mathbb{E}^{\Lambda_{s}}...\mathbb{E}^{\Lambda_{k}}G\right)\right)
\\ =  &C_2^{k-s}C_1\mu\left(\sum_{i\in N } H( \nabla_{i} f )\left(\mathbb{E}^{\Lambda_{s}}...\mathbb{E}^{\Lambda_{k}}G\right)\right)\end{align*}
 for constants $C_1=D+2d\eta$ and $C_2=(2d\eta)^\frac{k-s-1}{k-s}<1$ for sufficiently small $J_0$.
\end{proof} For a convex function $H:[0,+\infty)\rightarrow\mathbb{R}^+$ define
$$\omega_H(x)=\sup_{t>0}\frac{H(tx)}{H(t)} ,\  x>0$$
The following remark presents some useful properties for $\omega_H$.
\begin{rem}\label{Young0}It can be shown that $\omega_H\geq \frac{H}{H(1)}$ and $\omega_H(0)=0$ as well as that $\omega_H$ is convex and satisfies $\omega_H(ab)\leq\omega_H(a)\omega_H(b)$ for $a,b\geq0$. Furthermore, if $\frac{H}{x^r}$ is non decreasing for $x>0$ and $r>1$ then so is the function $\frac{\omega_H}{x^r}$. \end{rem}
\begin{Lemma}\label{lem2.2}Assume that the local specification $\{\mathbb{E}^{\Lambda,\omega}\}_{\Lambda \subset\subset \mathbb{Z}^d,\omega\in \Omega}$  is  such that   (H0), (H1)  and (H2) hold.
Then,  for sufficiently small $J_0$, for every $F:\mathbb{R}^N\rightarrow\mathbb{R}$,  for $N \subset \subset \mathbb{Z}^d$, such that   $\sum_{i\in N}H( \nabla_{i} F )\leq a$ $\mu-$a.e, there exist constants $C_1>0$ and $0<C_2<1$ such that     \begin{align*}\mu&\left( \frac{1}{\mathbb{E}^{\Lambda_{k+1}}e^{\lambda \mathcal{B}^{k,s}F}}Ent_{\mathbb{E}^{\Lambda_{k+1}}} (e^{\lambda \mathcal{B}^{k,s}F})\right)\leq ac\omega_H(\frac{\lambda}{2})C_{1} C_2^{k-s}  \end{align*}
for any   $k\geq s\in \mathbb{N}$  such that $N\subset \Lambda_{s-1}\cup\Lambda_{s}$ and  $\lambda \geq 0$.  \end{Lemma}
\begin{proof}  Since interactions occur only between nearest neighbours on the lattice,  the measure $\mathbb{E}^{\Lambda_{k+1},\omega}$ is the product measure of the single site measures i.e.  $\mathbb{E}^{\Lambda_{k+1},\omega}=\otimes _{j \in \Lambda_{k+1}}\mathbb{E}^{j,\omega}$.
Moreover, by (H1), all measures $\mathbb{E}^{j,\omega}, j \in\Lambda_{k+1}$ satisfy the $(MLS(H_\Phi))$ inequality with a constant $c$ uniformly on the
boundary conditions.   Since the $(MLS(H_\Phi))$ inequality is stable under tensorisation (see Remark  \ref{remark0}),  the product measure $\mathbb{E}^{\Lambda_{k+1},\omega}$ also satisfies the $(MLS(H_\Phi))$ inequality with the same constant $c$, therefore we have that   
$$\mathbb{E}^{\Lambda_{k+1}} \left\vert f\right\vert^2 log\frac{\left\vert f\right\vert^2}{\mathbb{E}^{\Lambda_{k+1}} \left\vert f\right\vert^2}\leq c\ \int \sum _{j\in\Lambda_{k+1}}H  \left( \frac{ \nabla_{j} f }{f}\right)f^2d\mathbb{E}^{\Lambda_{k+1}} $$
Denote $h=\mathcal{B}^{k,s}F$. If we plug $f:=\frac{e^{\frac{\lambda}{2}h}}{(\mathbb{E}^{\Lambda_{k+1}}e^{\lambda h})^\frac{1}{2}}$   we obtain
\begin{align*}\frac{Ent_{\mathbb{E}^{\Lambda_{k+1}}} (e^{\lambda h})}{\mathbb{E}^{\Lambda_{k+1}}e^{\lambda h}}&\leq c\ \int \sum _{j\in\Lambda_{k+1}}H  \left( \frac{ \lambda }{2}\nabla_{j} h\right)\frac{e^{\lambda h}}{\mathbb{E}^{\Lambda_{k+1}}e^{\lambda h}}d\mathbb{E}^{\Lambda_{k+1}} \\ &\leq c\omega_H(\frac{\lambda}{2}) \int \sum _{j\in\Lambda_{k+1}}H  \left( \nabla_{j} h\right)\frac{e^{\lambda h}}{\mathbb{E}^{\Lambda_{k+1}}e^{\lambda h}}d\mathbb{E}^{\Lambda_{k+1}}
\end{align*}
  If we apply the Gibbs measure in the last inequality we  get \begin{align*}\mu&\left( \frac{1}{\mathbb{E}^{\Lambda_{k+1}}e^{\lambda h}}Ent_{\mathbb{E}^{\Lambda_{k+1}}} (e^{\lambda h})\right)\\  &  \   \    \   \    \    \    \   \   \   \   \    \   \leq c\omega_H(\frac{\lambda}{2}) \int \sum _{j\in\Lambda_{k+1}}H  \left( \nabla_{j} (\mathcal{B}^{k,s}F)\right)\frac{e^{\lambda h}}{\mathbb{E}^{\Lambda_{k+1}}e^{\lambda h}}d\mu
\end{align*}
In order to calculate the right hand side we can use  Lemma \ref{lem2.1}. This leads to    \begin{align*}\mu&\left( \frac{1}{\mathbb{E}^{\Lambda_{k+1}}e^{\lambda \mathcal{B}^{k,s}F}}Ent_{\mathbb{E}^{\Lambda_{k+1}}} (e^{\lambda \mathcal{B}^{k,s}F})\right)\leq \\ & c\omega_H(\frac{\lambda}{2}) C_2^{k-s}C_1 \int \sum_{i\in N} H( \nabla_{i} F)\left(\mathbb{E}^{\Lambda_{s}}...\mathbb{E}^{\Lambda_{k}}\left(\frac{e^{\lambda  \mathcal{B}^{k,s}F}}{\mathbb{E}^{\Lambda_{k+1}}e^{\lambda \mathcal{B}^{k,s}F}}\right)\right)d\mu
\end{align*}Since  $\sum_{i\in N}H( \nabla_{i} F )\leq a$ $\mu-$a.e. we obtain  \begin{align*}\mu&\left( \frac{1}{\mathbb{E}^{\Lambda_{k+1}}e^{\lambda \mathcal{B}^{k,s}F}}Ent_{\mathbb{E}^{\Lambda_{k+1}}} (e^{\lambda \mathcal{B}^{k,s}F})\right)\\ &  \   \    \   \    \    \    \   \   \   \   \    \   \leq  ac\omega_H(\frac{\lambda}{2}) C_2^{k-s}C_1 \int \mathbb{E}^{\Lambda_{s}}...\mathbb{E}^{\Lambda_{k}}\left(\frac{e^{\lambda f}}{\mathbb{E}^{\Lambda_{k+1}}e^{\lambda f}}\right) d\mu \\ &  \   \    \   \    \    \    \   \   \   \   \    \   =ac\omega_H(\frac{\lambda}{2}) C_2^{k-s}C_1 \int \frac{e^{\lambda f}}{\mathbb{E}^{\Lambda_{k+1}}e^{\lambda f}}d\mu\\ &  \   \    \   \    \    \    \   \   \   \   \    \   =ac\omega_H(\frac{\lambda}{2}) C_2^{k-s}C_1 \int \mathbb{E}^{\Lambda_{k+1}}\frac{e^{\lambda f}}{\mathbb{E}^{\Lambda_{k+1}}e^{\lambda f}}d\mu\\ &  \   \    \   \    \    \    \   \   \   \   \    \   =ac\omega_H(\frac{\lambda}{2})C_{1} C_2^{k-s}
\end{align*}
where above we used successively the definition of the  Gibbs measure. 
  \end{proof} 
\begin{Lemma}\label{lem2.3}
For every function  $f$ for which the inequality
\begin{align}\label{eq1lem2.3}\mu\left( \frac{1}{\mathbb{E}^{\Lambda_{k}}e^{\lambda f}}Ent_{\mathbb{E}^{\Lambda_{k}}} (e^{\lambda f})\right)\leq K_f\omega_H(\frac{\lambda}{2})
\end{align} holds for  constants $K_f$ depending on $f$ we obtain that \begin{align*}\mu\left( \mathbb{E}^{\Lambda_k}\left( \{f-\mathbb{E}^{\Lambda_{k}}f\geq r\} \right)\right)\leq e^{-K_{f}\omega^*_H(\frac{2r}{K_{f}})}
\end{align*}
for any $r\geq0$ and  $\lambda \geq 0$.
\end{Lemma}
\begin{proof}
If we define  $\Psi(\lambda)=\mathbb{E}^{\Lambda_{k}}e^{\lambda f}$ then
we can write $$\mu\left( \frac{1}{\mathbb{E}^{\Lambda_{k}}e^{\lambda f}}Ent_{\mathbb{E}^{\Lambda_{k}}} (e^{\lambda f})\right)=\mu\left(\frac{ \lambda\Psi'(\lambda)}{\Psi(\lambda)}-\log\Psi(\lambda)\right)$$
Hence, due to (\ref{eq1lem2.3})
we have $$\mu\left(\frac{ \lambda\Psi'(\lambda)}{\Psi(\lambda)}-\log\Psi(\lambda)\right)\leq \omega_H(\frac{\lambda}{2})K_f$$ If we divide by $\lambda^2$
$$\mu \left( \frac{d}{d\lambda}\left(\frac{\log \Psi (\lambda)}{\lambda}\right) \right)\leq\frac{\omega_H(\frac{\lambda}{2})K_f}{\lambda^2}$$
Since $\lim_{\lambda\rightarrow 0}\frac{\log\Psi(\lambda)}{\lambda}=\mathbb{E}^{\Lambda_{k}}f$ by integrating we get
$$\mu\left( \int e^{\lambda (f-\mathbb{E}^{\Lambda_{k}}f)}d\mathbb{E}^{\Lambda_{k}} \right)\leq \exp \left\{ K_{f}\lambda \int_0^\lambda\frac{\omega_H(\frac{u}{2})}{u^2}du\right\}$$
By Chebichev inequality for any $r\geq0$ and any $\lambda>0$ we obtain
$$\mathbb{E}^{\Lambda_{k}}\left( \{f-\mathbb{E}^{\Lambda_{k}} f\geq r\}\right)\leq e^{-r\lambda}\int e^{\lambda (f-\mathbb{E}^{\Lambda_{k}}f)}d\mathbb{E}^{\Lambda_{k}}$$
If we combine together the last two inequalities we finally get
$$\mu\left(\mathbb{E}^{\Lambda_{k}}\left( \{f-\mathbb{E}^{\Lambda_{k}} f\geq r\}\right)\right)\leq  \exp \left\{-K_{f}\sup_{\lambda>0}\left[ \frac{r\lambda}{K_{f}}-\lambda\int_0^\lambda\frac{\omega_H(\frac{u}{2})}{u^2}du\right]\right\}$$
Since   $\frac{H(x)}{x^2}$ is non-decreasing on $(0,+\infty)$ we have that (see [B-R2]) $$\lambda\int_0^\lambda\frac{\omega_H(\frac{u}{2})}{u^2}du\leq\omega_H(\frac{\lambda}{2})$$
from which the result follows. 
\end{proof}
From Lemma \ref{lem2.2} and Lemma \ref{lem2.3} the  corollary bellow follows.
\begin{cor}\label{cor2.4} Assume that the local specification $\{\mathbb{E}^{\Lambda,\omega}\}_{\Lambda \subset\subset \mathbb{Z}^d,\omega\in \Omega}$  is  such that   (H0), (H1)  and (H2) hold.
Then,  for sufficiently small $J_0$, we have that for every $F:\mathbb{R}^N\rightarrow\mathbb{R}$,  for $N \subset \subset \mathbb{Z}^d$, such that   $\sum_{i\in N}H( \nabla_{i} F )\leq a$ $\mu-$a.e, there exist constants $C_1>0$ and $0<C_2<1$ such that   \begin{align*}\mu\left( \mathbb{E}^{\Lambda_{k+1}}\left( \{\mathcal{B}^{k,s}F-\mathcal{B}^{k+1,s}F\geq r\} \right)\right)\leq e^{-acC_{1} C_2^{k-s}\omega^*_H(\frac{2r}{acC_{1} C_2^{k-s}})}
\end{align*} 
for every $r\geq0$ and for any   $k\geq s\in \mathbb{N}$  such that $N\subset \Lambda_{s-1}\cup\Lambda_{s}$. \end{cor}Before we continue some useful facts about Orlicz spaces will be presented in the following two lemmas. These are more detailed presented in [B-R2], while for a thorough investigation one can look at [R-R]. 
\begin{Lemma}\label{Young1}Let $a\in (1,+\infty)$ and $\Phi$ be a differentiable, strictly convex nice Young function. Then the following are equivalent
\begin{enumerate}
\item
The function $\frac{\Phi}{x^a}$ is non-decreasing for $x>0$.
\item The function $\frac{\Phi^*}{x^{a^*}}$ is non-increasing for $x>0$.
\end{enumerate}
where $a^*$ is the conjugate of $a$, i.e. $\frac{1}{a}+\frac{1}{a^*}=1$.
\end{Lemma}
\begin{Lemma} \label{Young2} Let $a\in (1,+\infty)$ and $\Phi$ be a differentiable function on $[0,+\infty)$ such that $\frac{\Phi}{x^a}$ is non-decreasing. Then for $x>0$ and $t\in (0,1]$ $$\Phi(tx)\leq t^a\Phi(x)$$
\end{Lemma}
Now we can show the main concentration inequality for the infinite dimensional Gibbs measure.
\begin{proposition}\label{prop2.5} Assume that the local specification $\{\mathbb{E}^{\Lambda,\omega}\}_{\Lambda \subset\subset \mathbb{Z}^d,\omega\in \Omega}$  is  such that   (H0), (H1)  and (H2). Then for $J_0$ small enough, we have that for every $F:\mathbb{R}^N\rightarrow\mathbb{R}$,  for $N \subset \subset \mathbb{Z}^d$, with $\sum_{i\in N}H(\nabla_iF)\leq a $ $\mu-$a.e. the following holds$$ \mu \left( \{F-\mu (F)\geq r\} \right)\leq    e^{-\frac{ac\ddot C}{2 }\omega^*_H\left( \frac{2r}{ac} \right)}$$ for every $J_0,a,r$ such that $\left(\frac{1}{2^{t^*}C_2^{t^*-1}}\right)^k\geq k+1$ and  $ac\ddot C\omega^*_H\left( \frac{2r}{ac} \right)\geq 8\ln2$ and constants $C_1>0$,  $0<C_2<1$ and $\ddot C=\frac{1}{2^{2t^{*}}C_1^{t^*-1} }$ for $t^*$ the conjugate of $t$, where  $t$ is as in (H2).  \end{proposition}
\begin{proof}  for any   $s\in \mathbb{N}$  such that $N\subset \Lambda_{s-1}\cup\Lambda_{s}$, we can  write
$$F-\mu F=F-\mathcal{B}^{s,s}F+\sum_{k=0}^{n}(\mathcal{B}^{k+s,s}F-\mathcal{B}^{k+s+1,s}F)+\mathcal{B}^{s+1+n,s}F-\mu F$$
for any $n\geq 1$. The following lemma will allow us to take the limit of $n$ to infinity.
 \begin{Lemma}  \label{lem254PapIP}
Suppose the local specification  $\{\mathbb{E}^{\Lambda,\omega}\}_{\Lambda\subset\subset \mathbb{Z}^d,\omega \in\Omega}$ is such that  {\rm \textbf{(H0)}} and {\rm \textbf{(H1)}}/{\rm \textbf{(H3)}}  and that $J_0$ is sufficiently small.  Then, for any $f:\mathbb{R}^N\rightarrow \mathbb{R}$, for $N\subset \subset \mathbb{Z}^d$, we have that $\mathcal{B}^{n,s}f$ converges
$\mu$-almost everywhere to $\mu f$,  where   $s\in[0,n]$ such that $N\subset \subset \Lambda_{s-1}\cup \Lambda_s$.  In particular, $\mu$ is unique.\end{Lemma} 
The proof of this lemma will be presented in section \ref{convergence}. Since from Lemma \ref{lem254PapIP} we have that $\lim _{n\rightarrow \infty}\mathcal{B}^{n,s}F=\mu F$ $\mu-$a.e., taking the limit above leads to $$F-\mu F=F-\mathcal{B}^{s,s}F+\sum_{k=0}^{\infty}(\mathcal{B}^{k+s,s}F-\mathcal{B}^{k+s+1,s}F)$$
$\mu-$a.e. Therefore we have \begin{align*}\{F-&\mu F<r\}\supseteq\ \left\{ F-\mathcal{B}^{s,s}F <\frac{r}{2}\right\} \cap \left\{\cap_{k=0}^{\infty}\{\mathcal{B}^{k+s,s}F-\mathcal{B}^{k+s+1,s}F<\frac{r}{2^{k+2}}\}\right\}\end{align*}
$\mu-$a.e., which  leads to  
\begin{align}\label{eq1prop2.5}\mu  \left( \{F-\mu (F)\geq r\} \right) \nonumber    \leq &    \mu \left(  \left\{ F-\mathcal{B}^{s,s}F \geq\frac{r}{2}\right\}  \right)\\   &+\sum_{k=0}^{\infty}\mu\left( \{\mathcal{B}^{k+s,s}F-\mathcal{B}^{k+s+1,s}F\geq\frac{r}{2^{k+2}}\}\right)\end{align}
Since $\sum_{i\in N}H(\nabla_iF)\leq a$ $\mu-$a.e. we can use Proposition 26 from [B-R2] to bound the first term and Corollary \ref{cor2.4} to bound the second term on the right hand side of (\ref{eq1prop2.5}). This leads to   \begin{align}\label{neq6-7nlem3.8}\mu \left( \{F-\mu (F)\geq r\} \right)\leq e^{-ac\omega^*_H\left( \frac{2r}{2ac} \right)}+ &\sum_{k=0}^{\infty} e^{-acC_{1} C_2^{k}\omega^*_H\left( \frac{2r}{2^{k+2}acC_{1} C_2^{k}} \right)} \end{align}
Since from hypothesis (H2) there exists a $t>2$ such that $\frac{H(x)}{x^t}$ is non-increasing on $(0,+\infty)$, from Lemma \ref{Young1} for $t^*>1$ the dual of $t$, we have that $\frac{H^*(x)}{x^{t^*}}$ is non-decreasing.  Then,  if we combine  Remark \ref{Young0} together with  Lemma  \ref{Young2}, we obtain that for any $\theta>1$ and $x\geq 0$ we have that $\omega_H^*(\theta x)\geq \theta^{t^*}\omega_H^*(x)$. This gives the following bounds
\begin{equation}\label{neq6-7nlem3.8+1-1NeqEquation}e^{-ac\omega^*_H\left( \frac{2r}{2ac} \right)}\leq e^{-\frac{ac}{2^{t^*}}\omega^*_H\left( \frac{2r}{ac} \right)}\end{equation}and
$$e^{-acC_{1} C_2^{k}\omega^*_H\left( \frac{2r}{2^{k+2}acC_{1} C_2^{k}} \right)}\leq e^{-\frac{ac}{2^{2t^{*}}C_1^{t^*-1} }\omega^*_H\left( \frac{2r}{ac} \right)\left(\frac{1}{2^{t^*}C_2^{t^*-1}}\right)^k}$$
 Since $C_2$ can be as small as possible for small enough $J_0$, if we choose $J_{0}$ small enough so that $\left (\frac{1}{2^{t^*}C_2^{t^*-1}}\right)^k\geq k+1$ for all $k\in \mathbb{N}$ and $a, r$ such that $\frac{ac}{2^{2t^{*}}C_1^{t^*-1} }\omega^*_H\left( \frac{2r}{ac} \right)\geq 2\ln2$ we then get 
\begin{equation}\label{neq6-7nlem3.8+1} e^{-acC_{1} C_2^{k}\omega^*_H\left( \frac{2r}{2^{k+2}acC_{1} C_2^{k}} \right)}   \leq e^{-\frac{ac}{2^{2t^*+1}C_1^{t^*-1} }\omega^*_H\left( \frac{2r}{ac} \right)}\frac{1}{2^{k+1}}\end{equation} If we plug (\ref{neq6-7nlem3.8+1-1NeqEquation}) and (\ref{neq6-7nlem3.8+1})  in (\ref{neq6-7nlem3.8}), for $r$  large enough such that  $\frac{ac}{2^{2t^{*}+1}C_1^{t^*-1} }\omega^*_H\left( \frac{2r}{ac} \right)\geq4\ln2$  we finally obtain  \begin{align*}\mu \left( \{F-\mu (F)\geq r\} \right)\leq &e^{-\frac{ac}{2^{t^*}}\omega^*_H\left( \frac{2r}{ac} \right)}+ \sum_{k=0}^{\infty} e^{-\frac{ac}{2^{2t^{*}+1}C_1^{t^*-1} }\omega^*_H\left( \frac{2r}{ac} \right)}\frac{1}{2^{k+1}}\\   \leq &e^{-\frac{ac}{2^{2t^{*}+2}C_1^{t^*-1} }\omega^*_H\left( \frac{2r}{ac} \right)} \\ &\end{align*}
 \end{proof} We can now prove  Theorem \ref{theorem1}.

~

\noindent
\textbf{\textit{Proof of Theorem \ref{theorem1}.}} The proof follows directly from  Theorem 26 and 27 of [B-R2] and  Proposition \ref{prop2.5}.  Let $A\subset\Omega$ with $\mu (A)\geq \frac{1}{2}$ and define $$F_A(x)=\inf_{z\in A}\sum_{i\in N}H^*(x_i-z_i)$$ for $x=(x_i)_{i\in N}$. It is shown in Theorem 27 of [B-R2] that for the function  $F=\min(F_A,r)$ for $r>0$ one has that
\begin{align}\label{eq1prop2.6}\sum_{i\in N}H(\nabla_i F)\leq \omega_H(2)r
\end{align}
$\mu-$a.e. If we choose $J_{0}$ sufficiently small such that $\left( \frac{1}{2^{t^*}C_2^{t^*-1}}\right)^k\geq k+1$ for all $k\in \mathbb{N}$ and  $a=\omega_H(2)r$ then from  Proposition \ref{prop2.5}
we obtain that\begin{align}\label{eq2prop2.6}\mu \left( \{F-\mu (F)\geq \frac{r}{2}\} \right)\leq e^{-\frac{r\omega_H(2)\ddot C}{2}\omega^*_H\left( \frac{1}{\omega_H(2)c } \right)}\end{align} for every $r\geq R=\left(\omega_H(2)\ddot C \omega^*_H\left( \frac{2}{\omega_H(2)c } \right)\right)^{-1}16\ln2$. The rest of the proof follows [B-R2]. Since     $F_A=0$ on $A$ we get $\int\left(\{F\geq r\}\right)\leq r(1-\mu (A))\leq \frac{r}{2}$, which implies that $\{F\geq r\}\subset \{F-\mu (F)\geq \frac{r}{2}\}$. This together with inequality (\ref{eq2prop2.6}) gives
\begin{align*}\mu \left(\{F\geq r\}\right)\leq \mu \left(\left\{F-\mu (F)\leq\frac{r}{2}\right\}\right)\leq e^{-\frac{r\omega_H(2)\ddot C}{2}\omega^*_H\left( \frac{1}{\omega_H(2)c } \right)}
\end{align*}
Then the result follows from the following observation
$$\{F<r\}=\{F_A<r\}\subset A+\left\{x:\sum_{i\in N}H^*(x_i)<r\right\}$$ \qed
\section{A Perturbation Result for  the Log-Sobolev  Inequality.}
In Theorem \ref{theorem1} it was discussed how concentration properties can be obtained for the infinite dimensional Gibbs measure under the  main hypothesis  that the one site measure  $\mathbb{E}^{i,\omega}$ satisfied a Modified Log-Sobolev inequality with a constant uniformly with respect to the boundary conditions $\omega$. 
In this section we are concerned with the stronger Log-Sobolev inequality but we relax the (H1) hypothesis to one about the boundary free measure  \begin{equation}\label{nu meas}\nu(dx_{i})=\frac{e^{-\phi(x_{i})}dx_i}{\int e^{-\phi(x_{i})}dx_i}\end{equation} In this way a perturbation result is going to be shown for non log-concave measures for which  [B-E] and [B-H] cannot be applied. As a matter of fact, we will show that having relaxed the hypothesis (H1), the Gibbs measure satisfies  concentration properties similar to the ones that are true in the case of the Modified Log-Sobolev  inequality instead of the usual concentration properties that are associated with the Log-Sobolev inequality.

    We assume
that the one dimensional without interactions (boundary-free)  measure $\nu$  satisfies an
(LS) inequality and we  determine  conditions under which,  the  infinite volume Gibbs measure associated with the  local specification  $\{\mathbb{E}^{\Lambda,\omega}\}_{\Lambda\subset \subset \mathbb{Z}^d,\omega\in \Omega}$ as in (\ref{locspec}) with interactions $\Vert \nabla_i\nabla_jV(x_i,x_j)\Vert_\infty<\infty$,  satisfies concentration properties similar with the ones on Theorem \ref{theorem1}.

 Concerning perturbation properties related to the Spectral Gap inequality in infinite dimensions the following theorem due to [G-R] has been shown.
\begin{theorem}\label{Theoremby[G-R]}\textbf{([G-R])} If the measures $\nu(dx_{i})=\frac{e^{-\phi(x_{i})}dx_i}{\int e^{-\phi(x_{i})}dx_i}$ satisfy the Spectral Gap  inequality, then  the local specification  $\{\mathbb{E}^{\Lambda,\omega}\}_{\Lambda\subset \subset \mathbb{Z}^d,\omega\in \Omega}$ as in (\ref{locspec}) with interactions $\Vert \nabla_i\nabla_jV(x_i,x_j)\Vert_\infty<\infty$  satisfies the Spectral Gap inequality \begin{equation}\label{SGLocSpec]}\mathbb{E}^{\Lambda,\omega}\left( f-\mathbb{E}^{\Lambda,\omega}f \right)^2\leq \mathfrak{G}\mathbb{E}^{\Lambda,\omega}\vert \nabla_{\Lambda}f\vert^2\end{equation} with constant  $\mathfrak{G}$ uniformly in $\Lambda$ and the boundary $\omega$.
\end{theorem}
This result is stronger than the corresponding inequality for the Gibbs measure, since (\ref{SGLocSpec]}) implies that the infinite dimensional Gibbs measure $\mu$ corresponding to the local specification $\{\mathbb{E}^{\Lambda,\omega}\}_{\Lambda\subset \subset \mathbb{Z}^d,\omega\in \Omega}$ satisfies the Spectral Gap inequality
\begin{equation*}\mu\left( f-\mu f \right)^2\leq \mathfrak{G}\mu\vert \nabla f\vert^2\end{equation*}
Here, although we don't eventually obtain the Log-Sobolev inequality for the infinite dimensional Gibbs measure, we show that concentration properties still hold true. However, these are weaker than the ones that hold  for the product measure associated with the Log-Sobolev inequality, and similar to the ones that the Gibbs measure satisfies under the hypothesis (H1) of a Modified Log-Sobolev inequality $MLS(H_\Phi)$ with $\Phi(x)=x^4$. We will work with the following hypotheses:
\begin{itemize}
\item[\textbf{(H0):}] The interaction $V$ is such that  $\left\Vert \nabla_i \nabla_j V(x_i,x_j) \right\Vert_{\infty}<\infty.$
 
\item[\textbf{(H3):}] The one dimensional single site measures $\nu(dx_{i})=\frac{e^{-\phi(x_{i})}dx_i}{\int e^{-\phi(x_{i})}dx_i}$ satisfy an (LS)  inequality with a constant $c$. 
\item[\textbf{(H4):}] $J_{i,j}\geq 0$, $V\geq 0$ and  $\exists\epsilon>0$ and $\check K>0$: $\mu (U_{i,\omega}^2)\leq \check K$, 
 where $U_{i,\omega}=\hat c\log\mathbb{E}^{ i,\omega}e^{\epsilon \tilde U_{i,\omega}}$ for $\tilde U_{i,\omega}=2c\sum_{j\sim i}\left\vert  \nabla_{i}V(x_{i},\omega_{j})
\right\vert^2+\frac{1}{J_0}{\sum_{j\sim i}V(x_{i},\omega_{j})}$.
 \end{itemize}
The main result follows.
\begin{theorem} \label{NStheorem1}Assume the local specification $\{\mathbb{E}^{\Lambda,\omega}\}_{\Lambda \subset\subset \mathbb{Z}^d,\omega\in \Omega}$ as in (\ref{locspec})  is such that    (H0) and (H4) hold and that $\nu$ satisfies (H3). Then, for sufficiently small $J_0$, the corresponding infinite dimensional Gibbs measure $\mu$ is unique and there exists $R>0$ such that  for  every Borel set $A\subset \Omega$ with $\mu(A)\geq\frac{1}{2}$
\begin{align*}1-\mu\left( A+\left\{ x: \  \sum_{i\in N}\vert x_i\vert^\frac{4}{3}<r \right\}  \right)\leq e^{-\hat Kr}  \text{,  \   for every\ } r>\max \left\{R,\frac{\vert N \vert}{\omega_{x^4}(2)}\right\}
\end{align*}
for any $N\subset \subset \mathbb{Z}^d$ and some $\hat K>0$.\end{theorem} As a consequence of this, the analogue of the (\ref{Talagrand}) Talagrand type inequality follows (see [B-R2]). \begin{cor}\label{NScor1}Assume the local specification $\{\mathbb{E}^{\Lambda,\omega}\}_{\Lambda \subset\subset \mathbb{Z}^d,\omega\in \Omega}$ as in (\ref{locspec})  is such that    (H0) and (H4) hold and that $\nu$ satisfies (H3).  Then, for sufficiently small $J_0$, the corresponding infinite dimensional Gibbs measure $\mu$ is unique and  for every Borel set $A\subset \Omega$ with $\mu(A)\geq\frac{1}{2}$ we have $$\mu \left(A+\sqrt{r}B_2+\frac{1}{\omega_{\vert x \vert^\frac{4}{3}}^{-1}(\frac{1}{r})}B_{\vert x\vert^\frac{4}{3}}\right)\geq 1-e^{-Cr}$$ for  $C>0$ and $r>\max \left\{R,\frac{\vert N \vert}{\omega_{x^4}(2)}\right\}$ for $\mu$  on $ N \subset \subset \mathbb{Z}^d$.\end{cor}
Before the proof is presented we discuss some examples of measures that satisfy the hypothesis of the theorem. We are interested primarily in measures that satisfy  hypothesis (H3), for which in addition   the one site measure $\mathbb{E}^{i,\omega}$ does not satisfy the Log-Sobolev inequality uniformly on the boundary conditions, because in that  case the stronger Theorem \ref{Theorem[I-P]} can be applied.  Since for $\phi$ convex or convex at infinity, the Log-Sobolev inequality for $\mathbb{E}^{i,\omega}$ can be obtained from [B-E], [B-L2] and [B-H], we focus on non log-concave measures that go beyond convexity at infinity.    In [B-Z] a theorem is presented that allows to produce a variety of measures which satisfy the (LS) inequality but which go beyond convexity at infinity (see Theorem 5.5 in [B-Z]). An example of such a measure is as follows
$$\phi(x)=x^{p}+\vert x \vert^{p-1-\delta}\cos x$$
for $p>2$ and $\delta\in(0,1)$. In addition, since the phase $\phi$ dominates the interactions (H4) also follows.

We first state a  useful corollary which comes as a direct consequence of   Theorem \ref{Theoremby[G-R]} and Remark \ref{remark0}.
\begin{cor}\label{NS3.[G-R]}\label{NScor3-5}If the measures $\nu(dx_{i})=\frac{e^{-\phi(x_{i})}dx_i}{\int e^{-\phi(x_{i})}dx_i}$ satisfy the (LS)  inequality with a constant $c$, then  the one site measures  $\{\mathbb{E}^{i,\omega}\}_{ i \in \mathbb{Z}^d,\omega\in \Omega}$ as in (\ref{locspec})  for which (H0) is true,  satisfy a Spectral Gap inequality, say with constant $\hat c$.
\end{cor}
Under the condition (H3) the following weak form of Log-Sobolev type inequality for the one site measure $\mathbb{E}^{i,\omega}$ can be shown.
\begin{Lemma}\label{NSlemma3.3} If $J_{i,j}\geq 0$, $V\geq 0$ and (H3) holds for  the local specification $\{\mathbb{E}^{\Lambda,\omega}\}_{\Lambda \subset\subset \mathbb{Z}^d,\omega\in \Omega}$ then for $J_0$ sufficiently small, there exists an $\hat R\geq 0$ such that the one site measure $\mathbb{E}^{i,\omega}$ satisfies the following inequality 
\begin{equation*}\nonumber \mathbb{E}^{ i,\omega}(f^2log\frac{f^2}{\mathbb{E}^{ i,\omega}f^2}) \leq \hat R\mathbb{E}^{ i,\omega}\left(\left\vert \nabla_{i} f
\right\vert^2\right)+J_0U_{i,\omega}\mathbb{E}^{ i,\omega} \left(\left\vert \nabla_{i} f
\right\vert^2\right) 
\end{equation*}
where $U_{i,\omega}= \hat c\log\mathbb{E}^{ i,\omega}e^{\epsilon \tilde U_{i,\omega}}$ for $\tilde U_{i,\omega}=2 c\sum_{j\sim i}\left\vert  \nabla_{i}V(x_{i},\omega_{j})
\right\vert^2+\frac{1}{J_0}{\sum_{j\sim i}V(x_{i},\omega_{j})}$.\end{Lemma}
\begin{proof} The proof of the lemma is based on a perturbation result by [A-S].  Working as in [A-S], if we perturbe the hamiltonian $-\phi$ of the   measure $\nu(dx_i)$ with the function $h^i=-\sum_{j\sim i}J_{i,j}V(x_{i},\omega_{j})$, then for $J_0$ sufficiently small we obtain the following Log-Sobolev type inequality   
\begin{align}\label{NSeq3.9lem3.3}\mathbb{E}^{ i,\omega}(f^2log\frac{f^2}{\mathbb{E}^{ i,\omega}f^2}) \leq\frac{2c}{\epsilon-J_{0}}\mathbb{E}^{ i,\omega}\left\vert \nabla_{i} f
\right\vert^2+J_0\mathbb{E}^{ i,\omega} f^{2}\log\mathbb{E}^{ i,\omega}e^{\epsilon\tilde U_{i,\omega}} \end{align}
Where in the calculation of [A-S] we have also taken under account that $h^i$ is non positive.  From  [R], the following estimate
of the entropy holds
\begin{align*}\mathbb{E}^{ i,\omega}(\left\vert f\right\vert^2\log\frac{\left\vert f\right\vert^2}{\mathbb{E}^{ i,\omega}\left\vert f\right\vert^2}) \leq &A\mathbb{E}^{ i,\omega}\left\vert f-\mathbb{E}^{ i,\omega}f \right\vert^{2} +
\mathbb{E}^{ i,\omega}\left\vert f-\mathbb{E}^{ i,\omega}f \right\vert^2\log\frac{\left\vert f-\mathbb{E}^{ i,\omega}f \right\vert^2}{\mathbb{E}^{ i,\omega}\left\vert f-\mathbb{E}^{ i,\omega}f \right\vert^2}\end{align*}
  for some positive constant $A$. If we  use (\ref{NSeq3.9lem3.3}) to bound the second term on the right hand side of the last inequality we will obtain
\begin{align}\nonumber \mathbb{E}^{ i,\omega}(f^2log\frac{f^2}{\mathbb{E}^{ i,\omega}f^2}) \leq&\frac{2c}{\epsilon-J_{0}}\mathbb{E}^{ i,\omega}\left\vert \nabla_{i} f
\right\vert^2+\left(J_0\log\mathbb{E}^{ i,\omega}e^{\epsilon\tilde U_{i,\omega}}+A\right)\mathbb{E}^{ i,\omega} \left\vert f-\mathbb{E}^{ i,\omega}f \right\vert^{2} \end{align}
  We can bound the last term on the right hand side with the use
of the Spectral Gap inequality from Corollary \ref{NScor3-5}. This gives\\  \begin{align}\nonumber \mathbb{E}^{ i,\omega}(f^2log\frac{f^2}{\mathbb{E}^{ i,\omega}f^2}) \leq\left(\frac{2c}{\epsilon-J_{0}}+\hat cA\right)\mathbb{E}^{ i,\omega}\left\vert \nabla_{i} f
\right\vert^2+\hat cJ_0\log\mathbb{E}^{ i,\omega}e^{\epsilon\tilde U_{i,\omega}}\mathbb{E}^{ i,\omega} \left\vert \nabla_{i} f
\right\vert^2  \end{align}and the lemma follows for appropriate constant $\hat R$. \end{proof}
\begin{Lemma}\label{NSlem3.6}Assume that the local specification $\{\mathbb{E}^{\Lambda,\omega}\}_{\Lambda \subset\subset \mathbb{Z}^d,\omega\in \Omega}$  is  such that   (H0),  (H3) and (H4) hold.
Then,  for sufficiently small $J_0$, for every $F:\mathbb{R}^N\rightarrow\mathbb{R}$,  for $N \subset \subset \mathbb{Z}^d$, such that   $\sum_{i\in N}\left \vert \nabla_{i} F \right\vert^4\leq a$ $\mu-$a.e,  there exist constants $M>0$ and $0<C_3<1$ such that    \begin{align*}\nonumber\mu\left(\frac{1}{\mathbb{E}^{\Lambda_{k+1}}e^{\lambda \mathcal{B}^{k,s}F}}Ent_{\mathbb{E}^{\Lambda_{k+1}}} (e^{\lambda \mathcal{B}^{k,s}F})\right)\leq a& M\left(\frac{\lambda}{2}\right)^2C_3^{k-s}\end{align*}
for any $a$ such that $ a\geq \vert N\vert   \   \mu  \  \text{a.e.}  $
and for  $\lambda \geq 0$ and $k\geq s \in \mathbb{N}$ such that $N\subset \subset \Lambda_{s-1}\cup \Lambda_{s}$.
\end{Lemma}
\begin{proof}  Since interactions occur only between nearest neighbours on the lattice,  the measure $\mathbb{E}^{\Lambda_{k+1}}$ is the product measure of the single site measures i.e.  $\mathbb{E}^{\Lambda_{k+1,\omega}}=\otimes _{j \in \Lambda_{k+1,\omega}}\mathbb{E}^{j,\omega}$.
   The following product rule holds for the entropy of product measures 
$$\mathbb{E}^{ \Lambda_{k+1},\omega}(f^2log\frac{f^2}{\mathbb{E}^{ \Lambda_{k+1},\omega}f^2})\leq\sum_{i\in\Lambda_{k+1}}\mathbb{E}^{ \Lambda_{k+1},\omega}\mathbb{E}^{ i,\omega}(f^2log\frac{f^2}{\mathbb{E}^{ i,\omega}f^2})$$
(see [A-B-C]). If we use the $(LS)$ type inequality shown in Lemma \ref{NSlemma3.3} we obtain
\begin{align*}\mathbb{E}^{ \Lambda_{k+1},\omega}(f^2log\frac{f^2}{\mathbb{E}^{ \Lambda_{k+1},\omega}f^2})\leq & \hat  R\sum_{i\in\Lambda_{k+1}}\mathbb{E}^{ \Lambda_{k+1},\omega}\mathbb{E}^{ i,\omega}\left\vert \nabla_{i} f
\right\vert^2\\& +J_0\sum_{i\in\Lambda_{k+1}}\mathbb{E}^{ \Lambda_{k+1},\omega}\left(U_{i,\omega}\mathbb{E}^{ i,\omega} \left\vert \nabla_{i} f
\right\vert^2\right)
\end{align*}
Denote $h=\mathcal{B}^{k,s} F$ for $k \geq s$. If we plug $f:=\frac{e^{\frac{\lambda}{2}h}}{(\mathbb{E}^{\Lambda_{k+1}}e^{\lambda h})^\frac{1}{2}}$   we get
    \begin{align*}\frac{1}{\mathbb{E}^{\Lambda_{k+1}}e^{\lambda h}}En&t_{\mathbb{E}^{\Lambda_{k+1}}} (e^{\lambda h})\leq \hat  R\left(\frac{\lambda}{2}\right)^2\sum_{i\in\Lambda_{k+1}}\mathbb{E}^{ \Lambda_{k+1},\omega}\mathbb{E}^{ i,\omega}\left(\left\vert \nabla_i h
\right\vert^2 \frac{e^{\lambda h}}{\mathbb{E}^{\Lambda_{k+1}}e^{\lambda h}}\right)\\&+J_0\left(\frac{ \lambda }{2}\right)^{2}\sum_{i\in\Lambda_{k+1}}\mathbb{E}^{ \Lambda_{k+1},\omega}\left(U_{i,\omega}\mathbb{E}^{ i,\omega} \left(\left\vert \nabla_i h
\right\vert^2\frac{e^{\lambda h}}{\mathbb{E}^{\Lambda_{k+1}}e^{\lambda h}}\right)\right)  \end{align*}
One should notice that  for  every $j:dist(j,N)>k-s$ we have $\begin{array}{c}
\underbrace{\{\sim \{...\sim \{\sim j\} \}} \\
k-s \text{ \ times} \\
\end{array} \cap N=\emptyset$. As a consequence, concerning the quantity  $\left\vert \nabla_i h
\right\vert^2=\left\vert \nabla_i (\mathcal{B}^{k,s} F)
\right\vert^2=\left\vert \nabla_i ( \mathbb{E}^{\Lambda_k}...\mathbb{E}^{\Lambda_s} F)
\right\vert^2$ for $i\in \Lambda_{k+1}$ and $F$ with variables on $N\subset \subset \Lambda_{s-1}\cup \Lambda_s$,  from Lemma \ref{lem252PapIP} and \ref{lem2.1} we obtain that  $\left\vert \nabla_j h
\right\vert^2=0$ for every $j:dist(j,N)>k-s$. As a consequence, if we define the set $N_{k-s}:= \{j\in \mathbb{Z}^d:dist(j,N)\leq k-s\}$ we can write     \begin{align*}\frac{1}{\mathbb{E}^{\Lambda_{k+1}}e^{\lambda h}}En&t_{\mathbb{E}^{\Lambda_{k+1}}} (e^{\lambda h})\leq \hat  R\left(\frac{\lambda}{2}\right)^2\sum_{i\in\Lambda_{k+1}
\cap N_{k-s}}\mathbb{E}^{ \Lambda_{k+1},\omega}\mathbb{E}^{ i,\omega}\left(\left\vert \nabla_i h
\right\vert^2 \frac{e^{\lambda h}}{\mathbb{E}^{\Lambda_{k+1}}e^{\lambda h}}\right)\\&+J_0\left(\frac{ \lambda }{2}\right)^{2}\sum_{i\in\Lambda_{k+1}\cap N_{k-s}}\mathbb{E}^{ \Lambda_{k+1},\omega}\left(U_{i,\omega}\mathbb{E}^{ i,\omega} \left(\left\vert \nabla_i h
\right\vert^2\frac{e^{\lambda h}}{\mathbb{E}^{\Lambda_{k+1}}e^{\lambda h}}\right)\right)  \end{align*}
 If we apply the Gibbs measure in the last inequality we  obtain
  \begin{align}\label{NSeq3.10lem3.3}\nonumber\int \frac{1}{\mathbb{E}^{\Lambda_{k+1}}e^{\lambda h}}&Ent_{\mathbb{E}^{\Lambda_{k+1}}} (e^{\lambda h})d \mu \leq \hat  R\left(\frac{\lambda}{2}\right)^2\sum_{i\in\Lambda_{k+1}\cap N_{k-s}}\int \left\vert \nabla_i h
\right\vert^2 \frac{e^{\lambda h}}{\mathbb{E}^{\Lambda_{k+1}}e^{\lambda h}}d \mu\\&+J_0\left(\frac{ \lambda }{2}\right)^{2}\sum_{i\in\Lambda_{k+1}\cap N_{k-s}}\int U_{i,\omega}\mathbb{E}^{ i,\omega} \left(\left\vert \nabla_i h
\right\vert^2\frac{e^{\lambda h}}{\mathbb{E}^{\Lambda_{k+1}}e^{\lambda h}}\right)d \mu \nonumber \\    :=& \hat  R\left(\frac{\lambda}{2}\right)^2  I_1+J_0\left(\frac{ \lambda }{2}\right)^{2}  I_2  \end{align}  
 For the first term on the right hand side of (\ref{NSeq3.10lem3.3}) we can use   Lemma \ref{lem2.1}. This leads to    \begin{align*}I_1  \leq & C_2^{k-s}C_1 \int \sum_{i\in N} \vert\nabla_{i} F\vert^2\left(\mathbb{E}^{\Lambda_{s}}...\mathbb{E}^{\Lambda_{k}}\left(\frac{e^{\lambda h}}{\mathbb{E}^{\Lambda_{k+1}}e^{\lambda h}}\right)\right)d\mu\\     \leq&\frac{1}{2}  C_2^{k-s}C_1 \int\sum_{i\in  N} \vert\nabla_{i} F\vert^{4}\left(\mathbb{E}^{\Lambda_{s}}...\mathbb{E}^{\Lambda_{k}}\left(\frac{e^{\lambda h}}{\mathbb{E}^{\Lambda_{k+1}}e^{\lambda h}}\right)\right)d\mu \\ & +  \frac{1}{2}C_2^{k-s}C_1  \vert N \vert\int \mathbb{E}^{\Lambda_{s}}...\mathbb{E}^{\Lambda_{k}}\left(\frac{e^{\lambda h}}{\mathbb{E}^{\Lambda_{k+1}}e^{\lambda h}}\right)d\mu
\end{align*} 
Since  $\sum_{i\in N}\vert\nabla_{i} F\vert^{4}\leq a$ $\mu-$a.e. as well as $  \vert N \vert\leq a$ we obtain  \begin{align}\nonumber\label{rev1+1+1NSeq3.10lem3.3}I_1 & \leq C_2^{k-s}C_1a \int \mathbb{E}^{\Lambda_{s}}...\mathbb{E}^{\Lambda_{k-1}}\mathbb{E}^{\Lambda_{k}}\left(\frac{e^{\lambda h}}{\mathbb{E}^{\Lambda_{k+1}}e^{\lambda h}}\right)d\mu  \\    &   =C_2^{k-s}C_1a
\end{align}  
For the second term on the right hand side of (\ref{NSeq3.10lem3.3}), since $U_{i,\omega}=c\log\mathbb{E}^{ i,\omega}e^{\epsilon \tilde U_{i,\omega}}$ is a function that does not depend on the variable $i$, we have
  \begin{align}\label{addnew314lemma37}I_2& \nonumber=\sum_{i\in\Lambda_{k+1}\cap N_{k-s}}\int U_{i,\omega}\mathbb{E}^{ i,\omega} \left(\left\vert \nabla_i h
\right\vert^2\frac{e^{\lambda h}}{\mathbb{E}^{\Lambda_{k+1}}e^{\lambda h}}\right)d \mu \\   &=\sum_{i\in\Lambda_{k+1}\cap N_{k-s}}\int U_{i,\omega} \left\vert \nabla_i (h)
\right\vert^2\frac{e^{\lambda h}}{\mathbb{E}^{\Lambda_{k+1}}e^{\lambda h}}d \mu  \end{align}In order to bound the last term we will use the following lemma.
\begin{Lemma} \label{lem252PapIP} If the local specification $\{\mathbb{E}^{\Lambda,\omega}\}_{\Lambda\subset\subset \mathbb{Z}^d,\omega \in
\Omega}$ is such that (H0) and (H1)/(H3) is true, then, for sufficiently small $J_0$, there exists a constant  $\eta\in(0,1)$ such that
for every $i\in\Lambda_{k+1}$ 
\begin{align*}
\left\vert \nabla_{i }
(\mathbb{E}^{\Lambda_{k}}f) \right\vert ^2 &\leq 2
\mathbb{E}^{\{\sim i\}\cap \Lambda_k}\left\vert \nabla_{i } f
\right\vert^2  +\eta\mathbb{E}^{\{\sim i\}\cap \Lambda_k}\left\vert
\nabla_{\{\sim i\}\cap \Lambda_k} f\right\vert^2\end{align*}
for $k\in \mathbb{N}$.\end{Lemma}
The proof of  Lemma   \ref{lem252PapIP} will be presented in section 4.  We can calculate the gradient on the right hand side of (\ref{addnew314lemma37}) by applying   the lemma  $k-s$ times. For  $i\in\Lambda_{k+1}$ we  get 
\begin{align*}\left\vert \nabla_{i }
(h) \right\vert ^2=
\left\vert \nabla_{i }
(\mathcal{B}^{k,s} F) \right\vert ^2 \leq&2\eta^{k-s-1}E^i E^{j^i_1}E^{j^i_2}...E^{j^i_{k-s}}\left\vert
\nabla_{\{\sim j^i_{k-s-2}\}\cap\Lambda_{s-1}\cap N} F\right\vert^2\\&+\eta^{k-s}E^i E^{j^i_1}E^{j^i_2}...E^{j^i_{k-s}}\left\vert
\nabla_{\{\sim j^i_{k-s-1}\}\cap\Lambda_{s}\cap N} F\right\vert^2\end{align*}
where  we define $E^if:=\mathbb{E}^{\{\sim i\}\cap \Lambda_k}\sum_{j^i_1\in\{\sim i\}\cap\Lambda_k}f$  and for every $m=2,...,k-s$ $$E^{j^i_m}f:=\mathbb{E}^{\{\sim j^i_{m-1} \}\cap \Lambda_{k-m+1}}\sum_{j^i_{m}\in\{\sim j^i_{m-1}\}\cap\Lambda_{k-m+1}}f$$
If we denote $\tilde E^i_{k-s}:=E^i E^{j^i_1}E^{j^i_2}...E^{j^i_{k-s}}$, (\ref{addnew314lemma37}) becomes  
 \begin{align*}I_2\leq &2\eta^{k-s-1}\sum_{i\in\Lambda_{k+1}\cap N_{k-s}}\int U_{i,\omega} \tilde E^i_{k-s}\left\vert
\nabla_{\{\sim j^i_{k-s-2}\}\cap\Lambda_{s-1}\cap N} F\right\vert^2\frac{e^{\lambda h}}{\mathbb{E}^{\Lambda_{k+1}}e^{\lambda h}}d \mu \\   &+\eta^{k-s}\sum_{i\in\Lambda_{k+1}\cap N_{k-s}}\int U_{i,\omega} \tilde E^i_{k-s}\left\vert
\nabla_{\{\sim j^i_{k-s-1}\}\cap\Lambda_{s}\cap N} F\right\vert^2\frac{e^{\lambda h}}{\mathbb{E}^{\Lambda_{k+1}}e^{\lambda h}}d \mu \end{align*}
Since $ab\leq \frac{1}{2}(a^2+b^2)$ the last can be bounded be
\begin{align} \label{NSeq3.12lem3.3}\nonumber I_2 \leq  &(\eta^{k-s-1}+\frac{\eta^{k-s}}{2})\sum_{i\in\Lambda_{k+1}\cap N_{k-s}}\int U^{2}_{i,\omega}\frac{e^{\lambda h}}{\mathbb{E}^{\Lambda_{k+1}}e^{\lambda h}}d \mu\\  \nonumber &+\eta^{k-s-1}\sum_{i\in\Lambda_{k+1}\cap N_{k-s}}\int \left(\tilde E^i_{k-s}\left\vert
\nabla_{\{\sim j^i_{k-s-2}\}\cap\Lambda_{2}\cap N} F\right\vert^2\right)^2\left(\frac{e^{\lambda h}}{\mathbb{E}^{\Lambda_{k+1}}e^{\lambda h}}\right)d \mu   \\ & \nonumber \\  &  +\frac{\eta^{k-s}}{2}\sum_{i\in\Lambda_{k+1}\cap N_{k-s}}\int \left(\tilde E^i_{k-s}\left\vert
\nabla_{\{\sim j^i_{k-s-1}\}\cap\Lambda_{1}\cap N} F\right\vert^2\right)^2\left(\frac{e^{\lambda h}}{\mathbb{E}^{\Lambda_{k+1}}e^{\lambda h}}\right)d\mu \nonumber \\   :=&  (\eta^{k-s-1}+\frac{\eta^{k-s}}{2})I_{21}+\eta^{k-s-1}I_{22}+\frac{\eta^{k-s}}{2}I_{23}\end{align}
   For the first term of (\ref{NSeq3.12lem3.3}), since  $U_{i,\omega}^2$ for $i\in \Lambda_{k+1}$ does not depend on the variables $j \in \Lambda_{k+1}$ $$\mathbb{E}^{\Lambda_{k+1}}\left(U^{2}_{i,\omega}\frac{e^{\lambda h}}{\mathbb{E}^{\Lambda_{k+1}}e^{\lambda h}}\right)=U^{2}_{i,\omega}\mathbb{E}^{\Lambda_{k+1}}\left(\frac{e^{\lambda h}}{\mathbb{E}^{\Lambda_{k+1}}e^{\lambda h}}\right)=U^{2}_{i,\omega}$$ This leads to  \begin{equation}\label{NSeq3.14lem3.3}I_{21}=\sum_{i\in\Lambda_{k+1}\cap N_{k-s}}\mu (U^{2}_{i,\omega})\leq  \check K (2d)^{k-s}\vert N \vert\leq  \check K (2d)^{k-s}a \end{equation}
where above we used the bound from hypothesis (H4) and that  $\vert N \vert \leq a$.
For the second term on the right hand side of (\ref{NSeq3.12lem3.3}), we have \begin{align*}I_{22}   \leq &(2^{2d})^{k-s}\sum_{i\in\Lambda_{k+1}\cap N_{k-s}}\int \tilde E^i_{k-s}\left\vert
\nabla_{\{\sim j^i_{k-s-2}\}\cap\Lambda_{2}\cap N} F\right\vert^4\left(\frac{e^{\lambda h}}{\mathbb{E}^{\Lambda_{k+1}}e^{\lambda h}}\right)d \mu \\   \leq  & (2^{2d}2d)^{k-s}\sum_{i\in N}\int \left\vert \nabla_{i} F
\right\vert^{4}\mathbb{E}^{\Lambda_s}...\mathbb{E}^{\Lambda_k}\left(\frac{e^{\lambda h}}{\mathbb{E}^{\Lambda_{k+1}}e^{\lambda h}}\right)d \mu
\end{align*} Since     $\sum_{i\in N}\vert\nabla_{i} F\vert^{4}\leq a$ $\mu-$a.e. we obtain   \begin{align}\label{NSeq3.13lem3.3}\nonumber I_{22} \nonumber &\leq a(2^{2d}2d)^{k-s}\int \mathbb{E}^{\Lambda_s}...\mathbb{E}^{\Lambda_k}\left(\frac{e^{\lambda h}}{\mathbb{E}^{\Lambda_{k+1}}e^{\lambda h}}\right)d \mu\\ &=a(2^{2d}2d)^{k-s}
\end{align} If we work with the same way for the third term we get \begin{equation}\label{NSeq3.13lem3.3++1}  I_{23}\leq a (2^{2d}2d)^{k-s}\end{equation}
If we plug  (\ref{NSeq3.14lem3.3}), (\ref{NSeq3.13lem3.3})   and (\ref{NSeq3.13lem3.3++1}) in (\ref{NSeq3.12lem3.3}) we have
\begin{align} \label{NSeq3.15lem3.3}I_2 \leq (\eta^{k-s-1}+\frac{\eta^{k-s}}{2}) (\check K (2d)^{k-s} +(2^{2d}2d)^{k-s} )a\end{align}
If we combine together (\ref{NSeq3.10lem3.3}), (\ref{rev1+1+1NSeq3.10lem3.3})  and (\ref{NSeq3.15lem3.3}) we finally obtain
\begin{align*}\nonumber\mu(\frac{1}{\mathbb{E}^{\Lambda_{k+1}}e^{\lambda h}}&Ent_{\mathbb{E}^{\Lambda_{k+1}}} (e^{\lambda h}))\leq \hat R\left(\frac{\lambda}{2}\right)^2C_2^{k-s}C_1a\\&
+J_0\left(\frac{ \lambda }{2}\right)^{2} (\eta^{k-s-1}+\frac{\eta^{k-s}}{2})(\check K (2d)^{k-s} +(2^{2d}2d)^{k-s} )a\end{align*}
The lemma follows for appropriate constant $M>0$ and $C_3<1$, since $\eta$ and $C_2$ can be as small as we like for  sufficiently small $J_0$.
\end{proof}
The rest of the proof of Theorem \ref{NStheorem1} follows  like  the proof of Theorem \ref{theorem1}  in the previous section. With the use of  Lemma \ref{lem2.3} and Lemma \ref{NSlem3.6}  the proof of the theorem follows for $r\geq R=\left(\omega_H(2)\ddot C \omega^*_H\left( \frac{2}{\omega_H(2)c } \right)\right)^{-1}8\ln2$ and $a=\omega_H(2)r \geq \vert N \vert$ for $H(x)=x^4$. 
\section{Proof of Convergence and Sweeping Out Relations.}\label{convergence}
In this section the sweeping out relations of Lemma \ref{lem252PapIP} and  Corollary \ref{cor3.2} are presented, together with the convergence of $\mathcal{B}^n f$ $\mu-$a.e. to the Gibbs measure $\mu$.
 
~

\noindent
\textbf{\textit{Proof of Lemma \ref{lem252PapIP}.}}   For any $k\in \mathbb{N}$ and $i\in \Lambda_{k+1}$ we can write
\begin{align*}
\left\vert \nabla_{i}
(\mathbb{E}^{\Lambda_{k}}f) \right\vert ^2 \leq
\left\vert \nabla_{i }
(\mathbb{E}^{\{\sim i\}\cap \Lambda_k}f) \right\vert ^2
\end{align*}
Since the integration in the probability measure $\mathbb{E}^{^{\{\sim i\}\cap \Lambda_k},\omega}(d X_{^{\{\sim i\}\cap \Lambda_k}})$ is over  $\{ \sim i\}\cap \Lambda_k=\{j\in\Lambda_k:j \sim i\}$ while  $i\in \{\sim \left\{\{\sim i\}\cap \Lambda_k\right\}\}$, the variable $x_i$ appears in the boundary conditions of the integral $\mathbb{E}^{^{\{\sim i\}}\cap \Lambda_k,\omega}$. So,   we obtain
\begin{align}\label{eq220PapIPb}\nabla_i \mathbb{E}^{^{\{\sim i\}\cap \Lambda_k},\omega}f=&\nabla_i \left(\frac{\int f e^{-H^{\{\sim i \}\cap \Lambda_k, \omega}}dX_{\{\sim i\}\cap \Lambda_k}} {\int e^{-H^{\{\sim i \}\cap \Lambda_k, \omega}}dX_{\{\sim i\}}}\right)\nonumber\\ \nonumber=&\mathbb{E}^{^{\{\sim i\}\cap \Lambda_k},\omega}\left(\nabla_i f\right)+\mathbb{E}^{^{\{\sim i\}\cap \Lambda_k},\omega}f(-\nabla_i H^{\{\sim i \}\cap \Lambda_k, \omega})\\ \nonumber&-\mathbb{E}^{^{\{\sim i\}\cap \Lambda_k},\omega}f\mathbb{E}^{^{\{\sim i\}\cap \Lambda_k},\omega}(-\nabla_i H^{\{\sim i \}\cap \Lambda_k, \omega})\\=&\nonumber
\mathbb{E}^{^{\{\sim i\}\cap \Lambda_k},\omega}\left(\nabla_i f\right)+\mathbb{E}^{^{\{\sim i\}\cap \Lambda_k},\omega}f(-\sum_{j\in\{\sim i\}\cap \Lambda_k}J_{ij}\nabla_i V(x_{i}, x_{j}))\\ \nonumber &-\mathbb{E}^{^{\{\sim i\}\cap \Lambda_k},\omega}f\mathbb{E}^{^{\{\sim i\}\cap \Lambda_k},\omega}(-\sum _{j\in\{\sim i\}\cap \Lambda_k}J_{ij}\nabla_i V(x_{i}, x_{j}))\\ \leq&
\mathbb{E}^{^{\{\sim i\}\cap \Lambda_k},\omega}\left(\nabla_i f\right)+J_{0}\mathbb{E}^{^{\{\sim i\}\cap \Lambda_k},\omega}\left(\vert f-\mathbb{E}^{\{\sim i\}\cap \Lambda_k}f\vert \vert \mathcal{U}_i\vert \right)
\end{align} 
where above we have denoted
\begin{equation}\label{eq220-1interPapIP} W_i=\sum_{j\in\{\sim i\}\cap \Lambda_k}\nabla_i V(x_{i}, x_{j}) \text{\; and \;} \mathcal{U}_i=W_i-\mathbb{E}^{\{ \sim i\}\cap \Lambda_k}W_i\end{equation}
Hence, for any $i\in \Lambda_{k+1}$ from  (\ref{eq220PapIPb}) we get
that\begin{align}\label{eq221PapIP}\nonumber\left\vert \nabla_{i }
(\mathbb{E}^{\Lambda_{k}}f) \right\vert ^2& \leq2
\left\vert \mathbb{E}^{\{\sim i\}\cap \Lambda_k}\nabla_{i } f
\right\vert
^2+
2J_{0}^{2}
\left\vert\mathbb{E}^{\{\sim i\}\cap \Lambda_k}(f-\mathbb{E}^{\{\sim i\}\cap \Lambda_k}f)\mathcal{U}_i \right\vert ^2\\ \leq 2 &
\mathbb{E}^{\{\sim i\}\cap \Lambda_k}\left\vert \nabla_{i } f
\right\vert^2+2J_{0}^2\mathbb{E}^{\{\sim i\}\cap \Lambda_k}\left|f-\mathbb{E}^{\{\sim i\}\cap \Lambda_k}f\right|^2\mathbb{E}^{\{\sim i\}\cap \Lambda_k}\left|\mathcal{U}_i\right|^2
\end{align}
using H\"older's inequality.  Since interactions occur only between nearest neighbours in the lattice, we have that no interactions occur between points of the set $\{\sim i\}=\{j:j \sim i\}$, which consists only of the $2d$  points in $\mathbb{Z}^d$ that are nearest neighbours of the point $i$ (notice that: $i\notin \{\sim i\}$).  In the case of (H1), since   the measure $\mathbb{E}^{\{\sim i\},\omega}$ is the product measure of  single site measures i.e.  $\mathbb{E}^{\{\sim i\}\cap \Lambda_k,\omega}=\otimes_{j \in\{\sim i\}\cap \Lambda_k}\mathbb{E}^{j,\omega}$, because of hypothesis  (H1) and Remark \ref{remark0}  we conclude that   the product measure $\mathbb{E}^{\{\sim i\},\omega}$ also  satisfies the  $(MLS(H_\Phi))$  inequality with the same constant $c$.  By Remark \ref{remark0}, it then follows that $\mathbb{E}^{\{\sim i\}\cap \Lambda_k,\omega}$  satisfies the Spectral Gap inequality with constant $c_0= \frac{c}{2}$.
In the case of (H3), from  Corollary \ref{NScor3-5} and Remark \ref{remark0} the same follows for a constant $\hat c$, s.t. $c_0\leq \hat c$.   Hence we have
\begin{equation}
\label{eq222PapIP}
\mathbb{E}^{\{\sim i\}\cap \Lambda_k}\left|f-\mathbb{E}^{\{\sim i\}\cap \Lambda_k}f\right|^2\leq \hat c\mathbb{E}^{\{\sim i\}\cap \Lambda_k}\left\vert
\nabla_{\{\sim i\}\cap \Lambda_k} f\right\vert^2
\end{equation}
If we use the Spectral Gap inequality again for the product measure $\mathbb{E}^{\{\sim i\}\cap \Lambda_k}$  we get \begin{align}
\label{eq223PapIP}
\mathbb{E}^{\{\sim i\}\cap \Lambda_k}\left|\mathcal{U}_i\right|^2=& \mathbb{E}^{\{\sim i\}\cap \Lambda_k}\left|W_i-\mathbb{E}^{\{\sim
i\}\cap \Lambda_k}W_i\right|^2\leq  \hat c\sum_{_{ j \in\{\sim i\}\cap \Lambda_k}}\mathbb{E}^{\{\sim i\}\cap \Lambda_k}\left\vert \nabla_{j} W_i\right\vert
^2 \nonumber\\
&\leq  \hat c\sum_{_{ j \in\{\sim i\}\cap \Lambda_k}}\mathbb{E}^{\{\sim i\}\cap \Lambda_k}\left\vert \nabla_{j} \nabla_iV(x_{i}, x_{j})\right\vert
^2\leq 2d \hat cM^2\end{align}
where $M=\|\nabla_i\nabla_j V(x_i, \omega_j)\|_\infty < \infty$ ( by hypothesis (H0) ).
If we combine \eqref{eq221PapIP}, \eqref{eq222PapIP} and \eqref{eq223PapIP} we obtain
\begin{align}\label{newequat4.6-4.7lem4.1}
\left\vert \nabla_{i }
(\mathbb{E}^{\Lambda_{k}}f) \right\vert ^2 &\leq 2
\mathbb{E}^{\{\sim i\}\cap \Lambda_k}\left\vert \nabla_{i } f
\right\vert^2  +2c2d \hat cM^2J^2_0\mathbb{E}^{\{\sim i\}\cap \Lambda_k}\left\vert
\nabla_{\{\sim i\}\cap \Lambda_k} f\right\vert^2\end{align}
Therefore, choosing $J_0$ sufficiently small so that  $ \eta=2c2d \hat cM^2J^2_0<1$, the lemma follows. \qed

~

\noindent
\textbf{\textit{Proof of Corollary \ref{cor3.2}.}} For any $k\in \mathbb{N}$ we can write
\begin{align*}
\left\vert \nabla_{\Lambda_{k+1} }
(\mathbb{E}^{\Lambda_{k}}f) \right\vert ^2 = \sum_{i\epsilon \Lambda_{k+1}}\left\vert \nabla_{i }
(\mathbb{E}^{\Lambda_{k}}f) \right\vert ^2  \leq\sum_{i\epsilon\Lambda_{k+1}}
\left\vert \nabla_{i }
(\mathbb{E}^{\{\sim i\}\cap \Lambda_k}f) \right\vert ^2
\end{align*}
If we use Lemma \ref{lem252PapIP} we then obtain
\begin{align}\label{verynew4.7-8}\nonumber 
\left\vert \nabla_{\Lambda_{k+1} }
(\mathbb{E}^{\Lambda_k}f) \right\vert ^2\leq & 2\sum_{i\in \Lambda_{k+1}}
\mathbb{E}^{\{\sim i\}\cap \Lambda_k}\left\vert \nabla_i f
\right\vert^2  + \eta\sum_{i\in \Lambda_{k+1}}\mathbb{E}^{\{\sim i\}\cap \Lambda_k}\left\vert
\nabla_{\{\sim i\}\cap \Lambda_k} f\right\vert^2 \\  \leq  & 2\sum_{i \in \Lambda_{k+1}} \mathbb{E}^{\{\sim i\}\cap \Lambda_k} \left\vert \nabla_{i } f
\right\vert^2+\eta\sum_{i\in \Lambda_k}\sum_{j\sim i}\mathbb{E}^{\{\sim j\}\cap \Lambda_k}\left\vert \nabla_i f
\right\vert^2 \end{align}
where  $\eta<1$, for $J_0$ sufficiently small. For $H$ as in (H2) the Corollary follows for appropriate constants $D$ and $\eta<1$ for $J_0$ sufficiently small. \qed

~

 \noindent 
\textbf{\textit{Proof of Lemma \ref{lem254PapIP}.}}
 Following [G-Z] (see also [Pa]) we will  show that in   $L_1(\mu)$ we have   $lim_{n\rightarrow \infty}\mathcal{B}^{n,s} f=\mu f$ for any $f:\mathbb{R}^N\rightarrow \mathbb{R}$, with $N\subset \subset \Lambda_{s-1}\cup \Lambda_s$.
  For $k\in\mathbb{N}$ we have that
 \begin{align}\mu\vert\mathbb{E}^{^{\Lambda_{n}} } f- \mathbb{E}^{^{\Lambda_{n+1}} }\mathbb{E}^{\Lambda_{n}} f\vert^2&=\mu\mathbb{E}^{^{\Lambda_{n+1}} }\vert\mathbb{E}^{^{\Lambda_{n}} } f- \mathbb{E}^{^{\Lambda_{n+1}} }\mathbb{E}^{^{\Lambda_{n}} } f\vert^2  \label{3.11}\leq  \hat c\mu\left\vert \nabla_{^{\Lambda_{n+1}} }(\mathbb{E}^{^{\Lambda_{n}} } f
)\right\vert^2\end{align}          
The last inequality due to the  fact that   the measures  $\mathbb{E}^{^{\Lambda_{n+1}} }$  satisfy  the Spectral Gap inequality with constants  independently of the boundary conditions, as explained in the proof of  Lemma \ref{lem252PapIP}. If we use relationship (\ref{verynew4.7-8}) we get
$$\mu\vert\mathbb{E}^{^{\Lambda_{n}} } f- \mathbb{E}^{^{\Lambda_{n+1}} }\mathbb{E}^{\Lambda_{n}} f\vert^2\leq  \hat c2\mu \vert \nabla_{\Lambda_{n+1}}f\vert^2+\hat cC_4\mu \vert \nabla_{\Lambda_{n}}f\vert^2$$  
where $C_4 =2d \eta$.  From the last inequality we obtain that for any   $n\in \mathbb{N}$, 
 \begin{align*}\mu\vert \mathcal{B}^{n+s+1,s}f- \mathcal{B}^{n+s,s} f\vert^2 &\leq  \hat c2 \mu \vert \nabla_{\Lambda_{n+s+1}}(\mathcal{B}^{n+s-1,s} f)\vert^2+\hat cC_4\mu \vert \nabla_{\Lambda_{n+s}}(\mathcal{B}^{n+s-1,s} f)\vert^2\\  &=\hat cC_4\mu \vert \nabla_{\Lambda_{n+s}}(\mathcal{B}^{n-1+s,s} f)\vert^2\end{align*}
 If we use  relationship (\ref{verynew4.7-8}) to bound the last expression we have the following 
\begin{equation} \label{BorelCant1} \mu\vert \mathcal{B}^{n+s+1,s}f- \mathcal{B}^{n+s,s} f\vert^2 \leq \hat cC_4^{n-1}\left(2\mu\left\vert \nabla_{\Lambda_{n+s+1}} f
\right\vert^2+C_4\mu\left\vert \nabla_{\Lambda_{n+s}} f
\right\vert^2\right)\end{equation}  If we define the sets $A_{n}=\{\vert \mathcal{B}^{n+s,s}f- \mathcal{B}^{n+s+1} f\vert \geq (\frac{1}{2})^n\}$ we obtain 
\begin{equation*}\mu (A_n)=\mu\left( \{\vert \mathcal{B}^{n+s,s}f- \mathcal{B}^{n+s+1,s} f\vert \geq (\frac{1}{2})^n\} \right)\leq 2^{2n}\mu\vert \mathcal{B}^{n+s,s}f- \mathcal{B}^{n+s+1,s} f\vert^2  \end{equation*}by Chebyshev inequality. If we use (\ref{BorelCant1})  to bound the last inequality we have 
\begin{equation*}\mu (A_n)\leq \frac{1}{C_4}(4C_4 )^{n}\hat c\left(2\mu\left\vert \nabla_{\Lambda_{n+s+1}} f
\right\vert^2+C_4\mu\left\vert \nabla_{\Lambda_{n+s}} f
\right\vert^2\right)
\end{equation*}
We can choose $J_0$ sufficiently small such that $4C_4<\frac{1}{2}$ in which case we get that \begin{equation*}\sum_{n=0}^\infty\mu (A_n)\leq\left(\sum_{n=0}^\infty\frac{1}{2^n} \right) \frac{1}{C_4}\hat c\left(2\mu\left\vert \nabla_{\Gamma_1} f
\right\vert^2+\tilde\eta\mu\left\vert \nabla_{\Gamma_0} f
\right\vert^2\right)<\infty
\end{equation*}   From the  Borel-Cantelli lemma, only finite number of the sets $A_n$ can occur, which implies that  the sequence $\{\mathcal{B}^{n,s} f\}_{n \in \mathbb{N}}$  is a Cauchy sequence and that it converges $\mu-$almost surely. Say $$\mathcal{B}^{n,s}f\rightarrow \theta (f) \   \   \   \   \   \mu-\text{a.e.}$$ We will first show that $\theta (f)$  is a constant, i.e. it does not depend on variables on $\Gamma_0$ or $\Gamma_1$. To show that, first notice that  $\mathcal{B}^{n,s}(f)$ is a function on $\Gamma_1$ and $\Gamma_0$ when $n$  is even and odd respectively, which implies that the limits  $$\theta_{o}(f)=\lim_{n \text{\ odd}, n\rightarrow \infty}\mathcal{B}^{n,s}f \text{ \ and \ }\theta_e(f)=\lim_{n \text{\ even}, n\rightarrow \infty}\mathcal{B}^{n,s}f$$ do not depend on variables on $\Gamma_0$ and $\Gamma_1$ respectively.    Since both the  subsequences $\{\mathcal{B}^{n,s}f\}_{n\text{\ even}}$ and $\{\mathcal{B}^{n,s}f\}_{n\text{\ odd}}$  converge to $\theta (f)$ $\mu-$a.e.  we have that $$\theta_{o}(f)=\theta(f)=\theta_{e}(f)$$
 which implies that $\theta (f)$ is a constant. From that we obtain that \begin{equation}\label{BorelCant6} \mu \left( \theta (f) \right)=\theta(f)
\end{equation} 
Since the sequence $\{ \mathcal{B}^{n,s} f\}_{n \in \mathbb{N}}$ 
 converges $\mu-$a.e, the same holds for  the sequence $\{ \mathcal{B}^{n,s} f-\mu \mathcal{B}^{n,s} f\}_{n \in \mathbb{N}}$. We have  \begin{equation}\label{BorelCant6.5}\lim_{n\rightarrow \infty}( \mathcal{B}^{n,s} f-\mu \mathcal{B}^{n,s} f)=\theta (f)-\mu\left( \theta(f) \right) =\theta (f)-\theta (f)=0\end{equation}
where above we used (\ref{BorelCant6}). On the other side, we also have 
 \begin{equation}\label{BorelCant7}\lim_{n\rightarrow \infty}( \mathcal{B}^{n,s} f-\mu \mathcal{B}^{n,s} f)=\lim_{n\rightarrow \infty}( \mathcal{B}^{n,s} f-\mu f)=\theta (f)-\mu (f)
\end{equation} 
From (\ref{BorelCant6.5}) and (\ref{BorelCant7}) we get that 
$$\theta (f)=\mu (f)$$
We finally obtain $$\lim_{n\rightarrow \infty} \mathcal{B}^{n,s} f=\mu f,  \  \   \mu \  \  a.e. $$
\qed 
 \section{Conclusion}
In the present work we have investigated concentration of measure properties for the infinite dimensional Gibbs measure for local specifications with interactions with bounded second derivatives. In particular, we saw that  the infinite dimensional Gibbs measure satisfies a concentration inequality similar to the one obtained recently by [B-R2] for the product measure. Furthermore, similar results where obtained in the case where the one site boundary-free measure satisfies a Log-Sobolev inequality. 
 
In this paper we have been concerned with local specifications that have interactions that grow slower than quadratic. Concerning the (LSq) inequality for the infinite dimensional Gibbs measure the problem discussed in [I-P] is also  addressed in [Pa] for the case where the interactions are such that $$\Vert \nabla_i \nabla_j V(x_i,x_j)\Vert_\infty=+\infty $$
It is interesting to investigate  whether the concentration inequality obtained in Theorem \ref{theorem1} will also hold in the case of interactions that increase faster than quadratic as in [Pa].

Furthermore, one can try to prove the stronger  result obtained in [I-P] for the Modified Log-Sobolev inequality in the case where the interactions increase slower than a quadratic. That is, to show that  for sufficiently small $J_0$, the corresponding infinite dimensional Gibbs measure $\mu$    satisfies the  inequality
\begin{equation*}\mu \left(\vert f\vert^2\log \frac{\vert f\vert^2}{\mu \vert f \vert^2} \right)\leq C \int H_\Phi\left(\frac{\nabla f}{f}\right)f^2d\mu
\end{equation*} for some positive constant $C$. Then the concentration inequality shown here will  follow as a consequence.
 
 Concerning the perturbation result related to the Log-Sobolev inequality presented in Section 3, it is interesting to try to extend the result to the case of ($LS_q$) inequalities with $1<q<2$. The difficulty of this lies mainly in the fact that the perturbation result of [G-R] about the Spectral Gap inequality (SG2) does not seem to apply directly in the case of ($SG_q$) inequalities for $1<q<2$. If this is shown then the rest of the proof will follow exactly as in the case of $q=2$. In this case the  Gibbs measure will satisfy (\ref{lastrel}) with $\Phi(x)=\vert x \vert^{2q}$.

\section*{Acknowledgments}
The author wishes to thank Prof P. Cattiaux and F. Barthe  for  introducing him to the subject.

 \bibliographystyle{alpha}
\markboth{\textbf{Bibliography}}{
\end{document}